\documentclass[10pt,a4paper]{article}
\usepackage[utf8]{inputenc}
\usepackage[T1]{fontenc}
\usepackage[english]{babel}
\usepackage{amsmath, amsfonts, amssymb, amsthm, fancyhdr}
\usepackage{color} 
\usepackage{stmaryrd}
\usepackage{mdwlist}
\usepackage{dsfont}
\usepackage{bm}
\usepackage[all]{xy}
\usepackage{csquotes}
\definecolor{Prune}{RGB}{99,0,60}
\usepackage{mdframed}
\usepackage{multirow} 
\usepackage{multicol} 
\usepackage{scrextend} 
\usepackage{tikz}
\usepackage{graphicx}
\usepackage[absolute]{textpos} 
\usepackage{colortbl}
\usepackage{array}
\usepackage{subcaption}
\usepackage{geometry}
\usepackage{fullpage}
\usepackage{bbold}
\theoremstyle{plain}

\usepackage{hyperref}

\usepackage{appendix}
\usepackage{nccmath}

\title{Volumes of moduli spaces of directed ribbon graphs and Cut-and-Join operators}

\newtheorem{prop}{Proposition}
\newtheorem{thm}{Theorem}
\newtheorem{lem}{Lemma}
\newtheorem{cor}{Corollary}
\newtheorem{rem}{Remark}

\newcommand\Lbord{L_{\partial}}

\newcommand\End{\text{End}}

\newcommand\C{\mathbb{C}}

\newcommand\R{\mathbb{R}}
\newcommand\Rp{\mathbb{R}_{\ge 0}}
\newcommand\Rpp{\mathbb{R}_{>0}}
\newcommand\Q{\mathbb{Q}}
\newcommand\Z{\mathbb{Z}}
\newcommand\N{\mathbb{N}}
\newcommand\Npp{\mathbb{N}_{>0}}

\newcommand\Aut{\text{Aut}}

\newcommand{\Pc}{\mathcal{P}}
\newcommand\Hom{\text{Hom}}

\newcommand\M{\mathcal{M}}
\newcommand\Mc{\mathcal{M}^{comb}}

\newcommand\id{\operatorname{Id}}

\newcommand{\Rt}{\tilde{R}}

\newcommand\pr{\text{pr}}
\newcommand\Kb{\bold{K}}
\newcommand\Zb{\bold{Z}}

\newcommand\Kbull{K^{\bullet}}

\newcommand\s{\mathfrak{s}}
\newcommand\lf{\mathfrak{l}}
\newcommand\cf{\mathfrak{c}}
\newcommand\tb{\bold{t}}

\newcommand\alphab{\boldsymbol{\alpha}}

\newcommand\A{\mathcal{A}}

\newcommand\Dh{\hat{\mathcal{D}}}
\newcommand\eb{\bold{e}}
\newcommand\ez{\eb_{\emptyset}}

\newcommand\Zbull{Z^{\bullet}}

\newcommand\Rot{\tilde{R}^{\circ}}

\newcommand\D{\mathcal{D}}

\newcommand\bord{\textbf{bord}}
\newcommand\bordo{\textbf{bord}^{\circ}}
\newcommand\bordot{\widetilde{\textbf{bord}}^{\circ}}
\newcommand\bordoct{\widetilde{\textbf{bord}}^{\circ,\cf}}
\newcommand\bordoc{\textbf{bord}^{\circ,\cf}}
\newcommand\bordos{\textbf{bord}^{\circ,\s}}
\newcommand\bordosc{\textbf{bord}^{\circ,\s,\cf}}
\newcommand\bordol{\textbf{bord}^{\circ}_{\lf}}

\newcommand\acyclol{\textbf{acycl}_{\lf}}
\newcommand\acyclogen{\textbf{acycl}^{*}}
\newcommand\acyclolp{\widetilde{\textbf{acycl}}_{\lf}}
\newcommand\acyclop{\widetilde{\textbf{acycl}}}
\newcommand\acyclo{\textbf{acycl}}

\newcommand\Mo{M^{\circ}}
\newcommand\Mol{M^{\circ}_{\lf}}

\newcommand\Go{\mathcal{G}^{\circ}}
\newcommand\Gol{\mathcal{G}^{\circ}_{\lf}}
\newcommand\G{\mathcal{G}}

\newcommand\Ro{R^{\circ}}

\newcommand\Met{\text{Met}}
\newcommand\rib{\text{rib}}

\begin{document}
\author{Simon Barazer}
\maketitle

\begin{abstract}
    In this paper, we investigate the algebraic structure underlying the acyclic decomposition introduced in \cite{barazer2021cuttingorientableribbongraphs}. This decomposition applies to directed metric ribbon graphs and enables the recursive computation of the volumes of their moduli spaces. Building on this, we define integral operators with these volumes and show that they satisfy a Cut-and-Join type equation. Furthermore, we demonstrate that a suitable specialization of these operators gives rise to a generating series for \textit{dessins d'enfants}. 
\end{abstract}

\section{Introduction}
\label{section_intro}
In this paper, we study \textit{directed ribbon graphs}, which were previously investigated in \cite{barazer2021cuttingorientableribbongraphs,yakovlev2022contribution}. These are graphs embedded in a surface with boundary \cite{kontsevich1992intersection}, equipped with a coherent orientation of the edges (Subsection \ref{subesection_directed_ribbon_graphs}). Equivalently, their duals are bipartite maps, and then, their boundary components are labeled by $\pm$. A surface with boundaries labeled by $\pm$ will be called a \textit{directed surface}.\\
In \cite{barazer2021cuttingorientableribbongraphs}, a decomposition of directed ribbon graphs into graphs with a single vertex was found, it induces a recursion formula for the volumes of moduli spaces of directed metric ribbon graphs. In the present work, we explore some algebraic aspects of this decomposition.\\
We focus specifically on the principal strata of these moduli spaces, which correspond to four-valent ribbon graphs. The general case will be presented in a forthcoming paper. The volume of the moduli space of directed four-valent ribbon graphs of genus $g$ with $n^\pm$ boundary components of type $\pm$ is a function:
\begin{equation*}
   V_{g,n^+,n^-}(L^+|L^-),
\end{equation*}
defined on the subset:
\begin{equation*}
\label{formula_Lambda_1}
\Lambda_{n^+,n^-}=\{(L^+,L^-)|\sum_j L_i^+=\sum_j L_j\}\subset\Rp^{n^+}\times\Rp^{n^-}.
\end{equation*}
These functions are continuous piece-wise polynomials. Here we do not study the functions $V_{g,n^+,n^-}$ directly; instead, we adopt the following point of view. The presence of two sets of variables naturally gives rise to a family of integration kernels:
\begin{equation*}
P \to \int_{L^-}\frac{V_{g,n^+,n^-}(L^+|L^-)P(L^-)}{n^-!}dL^-.
\end{equation*}
The key ingredient is Proposition \ref{prop_transfert_lemma}, which states that these operators act on the space of symmetric polynomials. This proposition follows from the biparticity of the dual of a directed ribbon graph. The factor $n^-!$ is to avoid over-counting. We prefer to use the functions 
\begin{equation*}
    K_{g,n^+,n^-}=\prod_iL_i^+ V_{g,n^+,n^-}(L^+|L^-).
\end{equation*}
 In this case, the operator $K_{g,n^+,n^-}$ has degree $4g-4+2n^++2n^-$, which is twice the opposite of the Euler characteristic. This allows us to consider the action of
\begin{equation*}
K_{n^+,n^-}=\sum_{g,n^+,n^-}q^{2g-2+n^+ +n^-}K_{g,n^+,n^-},
\end{equation*}
which defines an operator
\begin{equation*}
 K_{n^+,n^-}: \hat{S}_{n^-}(V)\longrightarrow \hat{S}_{n^+}(V),
\end{equation*}
where $\hat{S}_n(V)$ is the space of symmetric formal series in $n$ variables and coefficients in $\Q$. Our plan is to sum over $n^+,n^-$ and form a single operator. To deal with the fact that there is an arbitrary number of boundary components, we consider the action on the \textit{formal Bosonic Fock space}:
\begin{equation*}
\hat{S}(V)=\prod_n \hat{S}_n(V).
\end{equation*}
All the formalism is recalled in Section \ref{subsection_boson}. In the formal Fock space it makes sense to consider the operator
\begin{equation*}
 K^{\cf \s}=\sum_{g,n^+,n^-}q^{2g-2+n^++n^-}K_{g,n^+,n^-},
\end{equation*}
where we sum over stable surfaces, i.e., $2g-2+n^++n^->0$. The formal Fock space is endowed with the symmetric product that will be denoted $\sqcup$; by using duality, this product is well defined for operators on $\hat{S}(V)$. The non-connected operator satisfies the classical formula:
\begin{equation*}
    K^{\s}=\exp_{\sqcup}(K^{\cf\s}).
\end{equation*}
It seems almost crucial to add the cylinder of type $(0,1,1)$, we define $K_{0,1,1}: \hat{S}_{1}(V)\to \hat{S}_1(V)$ to be the identity. Finally, we will be interested in
\begin{equation*}
K=\exp_{\sqcup}(\sum_{g,n^+,n^-} q^{2g-2+n^+ +n^-} K_{g,n^+,n^-}),
\end{equation*}
where we include $(0,1,1)$ in the sum. As we will see, the operator $K$ can be written as
\begin{equation*}
\label{formula_K_id_intro}
K=K^{\s}\sqcup \exp_{\sqcup}(K_{0,1,1})=K^{\s}\sqcup \id.
\end{equation*}
It corresponds to adding an arbitrary number of cylinders. This allows much more gluing, and it leads to a simple expression of the recursion for the functions $K_{g,n^+,n^-}$. A choice of basis allows us to identify $\hat{S}(V)$ and $\Q[[\tb]]$, where $\tb=(t_0,t_1,...)$ is an infinite set of variables. Formula \ref{formula_K_id_intro}
implies that $K$ is a formal series written in terms of \textit{creation} and \textit{annihilation} operators, and then, it acts on $\Q[[\tb]]$ as a formal differential operator (see Proposition \ref{prop_unstable_diff_operators}).\\

\begin{rem}
\label{remark_intro_M}
    A different point of view is to consider the algebra $S(\M)$ generated by directed surfaces. We can defines $K_{\Mo}$ and obtain a linear map $\Kb: S(\M)\to \End(\hat{S}(V))$, moreover, $\sqcup$ is well defined on $\M$ and corresponds to the usual disjoint union of surfaces. In this setting the map $\Kb$ is a morphism of commutative algebra.
\end{rem}

\paragraph{Cut-and-Join equation}
The main result of this paper is to rewrite the recursion given in \cite{barazer2021cuttingorientableribbongraphs}. We consider the following \textit{Cut-and-Join} operator: 
\begin{equation}
\label{formula_P_intro}
   W_1=\frac{1}{2}\sum_{i,j}(i+1)(j+1)t_{i+1}t_{j+1}\partial_{i+j} +\frac{1}{2}\sum_{i,j}(i+j+2)t_{i+j+2}\partial_i \partial_j.
\end{equation}
Cut-and-Join operators appear at several places in enumerative geometry \cite{goulden1997transitive,alexandrov2011cutkontsevitch,kazarian2015virasoro} and since \cite{alexandrov2022cut} we know that the \textit{topological recursion} implies a Cut-and-Join representation of the partition function. In practice, the Cut-and-Join equation is an efficient way to perform computation.
\begin{thm}
    The operator $K$ satisfies the linear equation:
    \begin{equation*}
      \frac{\partial K}{\partial q} = W_1K,~~\text{and}~~K(0)=\id.
    \end{equation*}
\end{thm}
Compared to the recursion for the volumes this formulation simplifies dramatically. Nevertheless, it is hard to go backward and obtain information on the Kernels. We give two different proofs: the first ones is straightforward and is obtained after rewriting the recursion in Theorem \ref{thm_rec_quad_K}. The second is more conceptual and uses the acyclic decomposition; it is a correspondence between directed ribbon graphs and \textit{acyclic stable graphs}. Generalizing Remark \ref{remark_intro_M}, we can form the space $S(\A)$ generated by acyclic stable graphs, for each stable graph $\Go$, we can define a kernel $K_{\Go}$. There is a way to compose acyclic stable graphs, and this gives rise to a  non-commutative product $\cdot$ on $S(\A)$, moreover, $\Kb$ extends to $S(\A)$ and $\Kb(x\cdot y)=\Kb(x)\circ\Kb(y)$. In this algebra, the exponential $\exp(P)$ of the two pairs of pants of type $(0,1,2)$ and $(0,2,1)$ corresponds to all the ways to glue pairs of pants together, and then, it is a generating series for acyclic stable graphs:
\begin{equation}
\label{formula_exp_A_intro}
    \exp(qW_1)=\sum_{\Go}\frac{q^{d(\Go)}n(\Go)}{\#\Aut(\Go)}\eb_{\Go}.
\end{equation}
Where we sum over all the possible acyclic stable graphs with components of type $(0,1,2)$ or $(0,2,1)$, $d(\Go)$ is the opposite of the Euler characteristic and $n(\Go)$ is the number of linear orders (see Proposition \ref{prop_exp_A}). Using the fact that $\Kb$ is a morphism of algebras, we can apply it to the both sides of Formula \ref{formula_exp_A_intro}; using the acyclic decomposition, the RHS is $K(q)$ and the LHS is equal to the operator given in Formula \ref{formula_P_intro}

\paragraph{Partition function}
It is quite natural to consider the partition function
\begin{equation*}
    Z(q,\tb)=K(q)\ez.
\end{equation*}
Where the \textit{vacuum} is given by $\ez=\exp(t_0)$. It appears that this series is a generating series for directed ribbon graphs or, equivalently, \textit{Grothendieck Dessins d'enfants} (see \cite{lando2004graphs} for instance), and is thus of combinatorial interest. A dessin d'enfant is a covering of the sphere ramified over three points $(x^+,x_0,x^-)$. The polynomials:
\begin{equation*}
    Z_{g,n^+,n^-}(L^+)=\int_{L^+} K_{g,n^+,n^-}(L^+|L^-)\frac{dL^-}{n^-!},
\end{equation*}
can be written as
\begin{equation*}
     Z_{g,n^+,n^-}(L^+)=\sum_{\alphab^+} \frac{(L^+)^{\alphab^+}}{\alphab^+!} h_{g,n^+,n^-}(\alphab^+).
\end{equation*}
The coefficients $ h_{g,n^+,n^-}(\alphab^+)$ count the number of Grothendieck dessins d'enfants with labeled ramifications over $x^+$ given by $\alphab^+=(\alpha_1^+,...,\alpha_{n^+}^+)$, simple unlabeled ramifications over $x_0$ and $n^+$ arbitrary unlabeled ramifications over $x^-$. It can be seen as an element of $\Q[\tb]$ and the partition function is given by
\begin{equation*}
    Z=\exp(\sum_{g,n^+,n^-} q^{2g-2+n^++n^-}Z_{g,n^+,n^-}).
\end{equation*}
Then, it is a generating series for a certain type of non-connected dessins d'enfants. From the definition of $Z$ we have the following corollary.
\begin{cor}
    The series $Z$ is the solution of the Cut-and-Join equation
    \begin{equation*}
        \frac{\partial Z}{\partial q}=W_1Z,~~\text{and}~~Z(0,\tb)=\exp(t_0).
    \end{equation*}
\end{cor}
It is also possible to obtain other classes of dessins d'enfants by considering different initial conditions but in this paper we restrict ourselves to the case where ramifications over $x_0$ are simple and we cannot control the ramification over $x^-$. The Cut-and-Join equation for this family of dessins d'enfants is not new, and we retrieve results of \cite{kazarian2015virasoro}.

\paragraph{Relation to topological recursion:}
It is natural to ask what is the relation with E.O topological recursion \cite{eynard2007invariants}. We consider $W_{g,n^+,n^-}$ the Laplace transform of $Z_{g,n^+,n^-}$ and defines
\begin{equation*}
    W_{g,n}=\sum_{n^-}W_{g,n,n^-}.
\end{equation*}
We include terms of types $(0,1)$ and $(0,2)$. 
\begin{prop}
\label{prop_loop_intro}
The series satisfies the loop equation:
    \begin{eqnarray*}
     x_1 W_{g,n} &=& \sum_{i\neq 1} \frac{\partial}{\partial x_i}\left( \frac{W_{g,n-1}(x_1,x_{\{1,i\}^c})-W_{g,n-1}(x_i,x_{\{1,i\}^c})}{x_1-x_i}\right)\\
&+&W_{g-1,n+1}(x_1,x_1,x_{\{1\}^c})+ \sum_{I_1,I_2,g_1+g_2=g} W_{g_1-1,n_1+1}(x_1,x_{I_1})W_{g_2-1,n_2+1}(x_1,x_{I_2})+\delta_{g,0}\delta_{n,1}.
    \end{eqnarray*}
\end{prop}
This equation is classical, in particular, the \textit{disc} amplitude is given by:
\begin{equation*}
    xW_{0,1}(x)=W_{0,1}(x)^2+1.
\end{equation*}
By using argument that are standard in the TR community and developed in \cite{Bertrand_Eynard_2004_hermetian_tr,eynard2007invariants}  we obtain that:
\begin{cor}
    The $W_{g,n}$ can be computed by the topological recursion using the spectral curve $xy=y^2+1$.
\end{cor}
Proposition \ref{prop_loop_intro}  also implies the following system of equations, which is called Virasoro constraints (or Witt constraints). Let $L_i$ the operator defined by
\begin{equation*}
 L_i= \partial_{i+2} - \sum_{j}(j+1) t_{j+1}\partial_{i+j+1}-\sum_{k+l=i} \partial_k\partial_l,
\end{equation*}
for $i\ge -1$.
\begin{cor}
    The operators $L_i$ form a representation of the Witt algebra:
    \begin{equation*}
        [L_i,L_j]=(i-j)L_{i+j}~~~~\forall i,j\ge -1.
    \end{equation*}
    And the series satisfies
    \begin{equation*}
        L_iZ=0~~~~~\forall i\ge -1
    \end{equation*}
\end{cor}
Similarly to the Cut-and-Join equation, the Virasoro constraints can be found in \cite{kazarian2015virasoro}.

\paragraph{Case of bivalent vertices}
In Section \ref{section_marked} we generalize the content of Sections \ref{section_operator_volumes}, \ref{section_cut_and_join_operator} and \ref{section_partition_function}  to the case of ribbon graphs with vertices of degree two. Although this case is relatively simple, it is the starting point of a forthcoming paper. We can consider $K_{g,n^+,^-,m}$ the volumes associated to moduli spaces of directed graphs with $2g-2+n^++n^-$ vertices of degree four and $m$ vertices of degree two. We consider:
\begin{equation*}
    K(q_0,q_1)=\exp_{\sqcup}(\sum_{g,n^+,n^-,m}q_0^m q_1^{2g-2+n^++n^-}K_{g,n^+,n^-,m}).
\end{equation*}
Specializing at $q_0=0$, we recover the precedent operator. We introduce the operator 
\begin{equation*}
    W_0=\sum_i (i+1)t_{i+1}\partial_i,
\end{equation*}
 We obtain the following Theorem:
\begin{thm}
    The two operators $W_0,W_1$ commute and:
    \begin{equation*}
        \frac{\partial K}{\partial q_i} = W_iK,~~\text{and}~~K(0)=id.
    \end{equation*}
\end{thm}
This is what is called $W-$representation, or at least the first two equations \cite{Mironov_2011,Wang_2022}. We can also consider the partition function; in this case, it counts the ribbon graphs with bivalent vertices or, equivalently, the ribbon graphs with integral metrics on the edges, but, we do not control the length of the negative boundary components.

\paragraph{Plan of the paper}
\begin{enumerate}
    \item In Section \ref{section_background} we recall some tools useful to the following section. Subsection \ref{paragraph_surface_background} is devoted to acyclic stable graphs. In subsection \ref{subsection_boson} we give a self-contained introduction to symmetric algebra and operators on it. We also introduce the algebra generated by acyclic stable graphs and write the formula that relates exponentials in this algebra to generating series.
    \item  In Section \ref{section_operator_volumes} we construct the operators; we start by recalling results on moduli space of directed stable graphs and prove Proposition \ref{prop_transfert_lemma} which is the starting point of our investigation. We study elementary properties of these operators.
    \item In Section \ref{section_cut_and_join_operator} we give two proofs of the Cut-and-Join equation.
    \item In Section \ref{section_partition_function} we study the partition function and find the combinatorial interpretation for the coefficients. We derive the Cut-and-Join equation and also study the case $n^-=1$.
    \item In Section \ref{section_TR_oriented} we obtain the Loop equation, these results can also be found in \cite{eynard2016counting} in a different way. 
    \item Finally, in section \ref{section_marked} we study the case of ribbon graphs with bivalent vertices and give several interpretations for the partition function.
    \item We include two appendices. In \ref{subsection_dual} we present a natural scalar product such that $V_{g,n^-,n^+}$ is the adjoint of $V_{g,n^+,n^-}$. In subsection \ref{subsection_norbury} we  give a relation between usual ribbon graphs and ribbon graphs with metrics on the edges that were studied by P. Norbury \cite{norbury2008counting}. We found a \textit{combinatorial} change of variable that gives a relation between the two generating functions and explains why these two different combinatorial problems share the same spectral curve.
\end{enumerate}

\paragraph{Acknowledgment:} 
The author warmly thanks numerous colleagues and friends for their support throughout this work. First and foremost, I am deeply grateful to my advisors, M. Kontsevich and A. Zorich, for their guidance and support during my PhD, and to the IHÉS for its hospitality. I also thank M. Liu, E. Goujard, V. Delecroix, and Y. Yakovlev for their valuable input and interest in this work.

\tableofcontents 
\section{General background and notations}
\label{section_background}

\subsection{Topological background}
\paragraph{Directed surfaces:}
\label{paragraph_surface_background}
Let $M$ be a compact, oriented, topological surfaces with a nonempty boundary. By abuse of notation, we denote $\partial M$ the set its boundary components, i.e., $\pi_0(\partial M)$. A \text{direction} on $M$ is a non constant function, $\epsilon : \partial M \longrightarrow \{\pm 1\}$ and the pair $\Mo=(M,\epsilon)$ is a \textit{directed surface}.\\
We often assume that $\Mo$ is stable, and denote by $\bordo$ the set of directed stable surfaces up to homeomorphisms. The subset of connected surfaces is denoted $\bordoc$. A direction divides $\partial M$ into two non-empty sets $\partial^\pm\Mo$. We denote $n^\pm = \#\partial^{\pm}\Mo$; a connected directed surface is then characterized by the triple $(g,n^+,n^-)$ where $g$ is the genus of the underlying surface. If $\Mo\in \bordo$ and $c\in \pi_0(\Mo)$, we denote $\Mo(c)\in \bordoc$ the corresponding surface and write $\Mo=\sqcup_{c}~\Mo(c)$. We also use the notation $d(\Mo)$ for the opposite of the Euler characteristic; if $\Mo$ is of type $(g,n^+,n^-)$ we have
\begin{equation*}
    d(\Mo)=2g-2+n^++n^-,
\end{equation*} 
moreover, the degree is additive under disjoint union \begin{equation*}
    d(\Mo_1\sqcup \Mo_2)=d(\Mo_1)+d(\Mo_2).
\end{equation*}
Two directed surfaces play a central role in what follow: the two pairs of pants $P_+$ and $P_-$, of respective type $(0,2,1)$ and $(0,1,2)$, both satisfying $d(P_{\pm})=1$. If $\Mo\in \bordoc$, we define
\begin{equation}
\label{formula_LambdaMo}
   \Lambda_{\Mo}=\left\{L=(l_\beta)\in \Rp^{\partial M}~|~{\sum}_\beta \epsilon(\beta)l_\beta=0 \right\},
\end{equation}
for a disconnected surface, if $\Mo=\sqcup_c~ \Mo(c)$ we set $\Lambda_{\Mo}=\prod_c \Lambda_{\Mo(c)}$. 

\begin{rem}[Labeling:]
\label{remark_label_1}
We consider surfaces with labeled boundaries; a \textit{labeling} on $\Mo$ is a pair of bijections $\lf^{\pm}: \partial^{\pm} \Mo\to \llbracket 1 , n^{\pm}\rrbracket$. These bijections define two partitions $(I^\pm(c))_c$ of $ \llbracket 1 , n^{\pm}\rrbracket$ indexed by the connected components $c$ of $\Mo$, we have $I^{\pm}(c)=\lf^{\pm}(\partial \Mo(c))$. If the surface is connected of type $(g,n^+,n^-)$, $\Lambda_{\Mo}$ can be identified  with the space $\Lambda_{n^{+},n^{-}}$ given in the introduction (see Formula \ref{formula_Lambda_1}).
\end{rem}

\paragraph{Directed stable graphs:}
\label{paragraph_directed_stable}
A \textit{directed stable graph} $\Go$ is the data of a directed surface $\tilde{M}^{\circ}_{\Go}$ with an involution $s_1: \partial \tilde{M}^{\circ}_{\Go}\to \partial \tilde{M}^{\circ}_{\Go}$, that satisfies $\epsilon(s_1\beta)=-\epsilon(\beta)$ for all boundary components $\beta$ that is not a fixed point of $s_1$. We denote $X\Go =\partial \tilde{M}^{\circ}_{\Go}$, the edges $X_1\Go$ are the orbits of $s_1$ of length two and the legs $\partial \Go$ are the fixed points. We assume that the set $\partial \Go$ is nonempty. The vertices $X_0\Go$ of $\Go$, corresponds to the connected components of $\tilde{M}^{\circ}_{\Go}$, for $c\in X_0\Go$ we denote $\Go(c)$ the corresponding connected components. \\
We can orient the edges of $\Go$ so that the source is positive and the target is negative (see Figure \ref{figure_acyclic_pant_decomp}). By forgetting the topology of the components, a directed stable graph defines a directed graph.\\
 A directed stable graph $\Go$ represents a family of directed surfaces glued together along boundary components that have opposite signs (To be consistent, we must assume that the graphs are non-degenerate).  We denote  $\Aut(\Go)$ the group of automorphisms of $\Go$ that fix $\partial \Go$. By gluing the boundary components identified by $s_1$, we obtain a stable directed surface $\Mo_{\Go}$.\\
\begin{figure}
    \centering
    \includegraphics[width=0.5\linewidth]{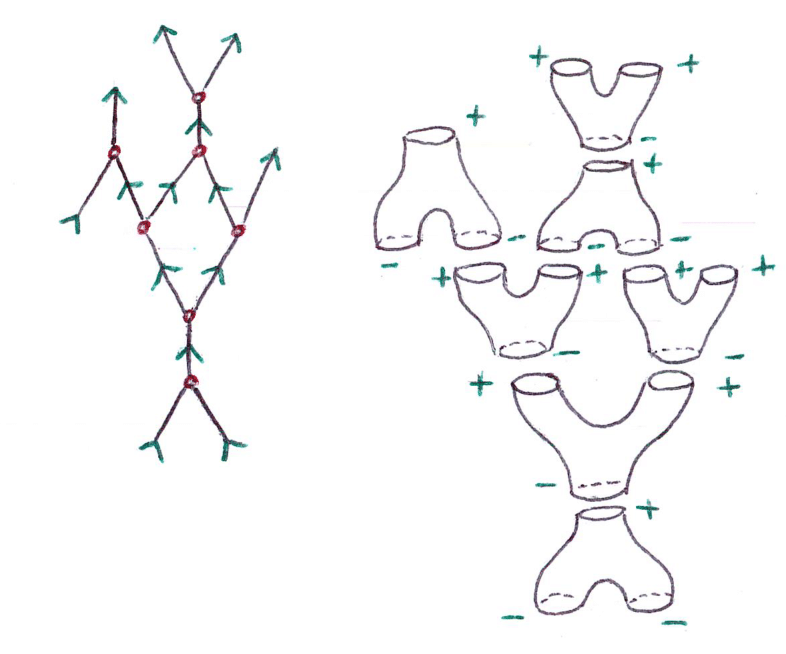}
    \caption{An acyclic directed stable graph.}
    \label{figure_acyclic_pant_decomp}
\end{figure}

\paragraph{Acyclic stable graphs:}
\label{paragraph_acyclic_stable}
A directed stable graph $\Go$ induces a relation on the set ofvertices.$X_0\Go$. We say that $x\ge y$ iff there is a path oriented positively from $y$ to $x$. A graph is \textit{acyclic} if there is no closed cycle, or equivalently, if this relation is a strict partial order on $X_0\Go$. A \textit{linear order} on a directed acyclic graph is an enumeration of the vertices, which is increasing for the order relation.\\
We denote $\acyclo$ the set of isomorphism classes of acyclic stable graphs and $\acyclol$ the labeled version. For a given directed surface $\Mo$, we denote $\acyclo(\Mo)$ the set of acyclic stable graphs $\Go$ with $\Mo\simeq \Mo_{\Go}$ and for each $(g,n^+,n^-)$ with $2g-2+n^++n^->0$ we denote $\acyclo_{g,n^+,n^-}$ the labeled acyclic stable graphs of this topological type. We also use the upper script $*$ to design acyclic stable graphs whose vertices are all pairs of pants of type $(0,1,2)$ and $(0,2,1)$. 
\paragraph{Cone of relative cycles:}
\label{paragraph_cone}
If $\Go=(\G,\epsilon)$ is a directed stable graph, we can consider the cone $\Lambda_{\Go}$ of directed cycles on $\Go$:
\begin{equation}
\label{conedirected_graph}
\Lambda_{\Go} = \{L\in \prod_c \Lambda_{\Go(c)}~|~l_\beta=l_{s_1(\beta)}~~~~\forall~\beta\in X\Go\}.
\end{equation}
For each $\gamma\in X_1\G$ the projection $l_\gamma : \Lambda_{\Go}\longrightarrow \Rp$ represents the length of $\gamma$, at each vertex of the graph, the signed sum of the lenghts is equal to $0$. If $\Go\in \acyclo(\Mo)$, we can identify $\partial \Go$ and $\partial \Mo$ andset it $\Lbord=(l_\beta)_{\beta\in \partial M}$, it defines a linear map:
\begin{equation*}
\Lbord : \Lambda_{\Go} \longrightarrow \Lambda_{\Mo}.
\end{equation*}
For all $L\in \Lambda_{\Mo}$, we denote $\Lambda_{\Go}(L)=\Lbord^{-1}(\{L\})$; if $L$ is in txhe image of $\Lbord$, the space $\Lambda_{\Go}(L)$ is a polytope that is endowed with a natural Lebesgue measure $d\sigma_{\Go}(L)$ normalized by a lattice of integer points in its tangent space. As $\Go$ is acyclic, the polytope $\Lambda_{\Go}(L)$ is bounded and we can define:
\begin{equation}
\label{formula_VG}
        V_{\Go}(L)=\int_{\Lambda_{\Go}(L)}\prod_\gamma l_\gamma ~d\sigma_{\Go}(L).
\end{equation}
It is a continuous piecewise polynomial function on $\Lambda_{\Mo}$. If $\Go$ is labeled, we can split $L=(L^+,L^-)$ and regard $V_{\Go}$ as a function $V_{\Go}(L^+|L^-)$ of two sets of variables.
\subsection{Generalities on \textit{Bosonic} Fock spaces}
\label{subsection_boson}
\paragraph{Multi-indices and partitions:}
\label{paragraph_tensor}

Let $A$ be a countable set. Defines the set of multi-indices $\mathcal{I}(A)= \bigsqcup_{n\ge 0} A^n$, where $\alphab=(\alpha_1,...,\alpha_n)$ has length $n(\alphab)=n$. By convention, $A^0$ contains a single element the empty index. There is a natural monoid structure $\otimes$ on $\mathcal{I}(A)$ given by concatenation.\\
A \textit{partition} $\mu \in \Pc(A)$ is a map $\mu: A \longrightarrow \N$ with finite support. Its length is $n(\mu)=\sum_{\alpha\in A} \mu(\alpha)$. The symmetric group $\frak{S}_n$ acts on the set $A^n$, and its orbits  are indexed by partitions $\mathcal{P}_n(A)$ of length $n$. Given $\alphab\in A^n$ let $\mu_{\alphab}\in \mathcal{P}_n(A)$ the partition given by $\mu_{\alphab}(\alpha)=\#\{i|\alpha_i=\alpha\}$. The orbit of $\alphab$ is the set $C(\mu_{\alphab})$ of all the $\alphab'$ with $\mu_{\alphab'}=\mu_{\alphab}$, we have
\begin{equation*}
 \#C(\mu)=\frac{n(\mu)!}{\mu!},~~~\text{with}~~~~~\mu! =\prod_\alpha \mu(\alpha)!.
\end{equation*}
There is a natural law $+$ on $\Pc(A)$ and the map $\alphab\to \mu_{\alpha}$ is additive.

\paragraph{Tensor and symmetric algebra:}
Let $V$ be the free vector space over $\Q$ generated by $A$ (in what follows $V=\Q[L]$ and $A=\N$). For $n\in \N$, we denote $T_n(V)=V^{\otimes n}$ and $T(V)=\bigoplus_n T_n(V)$. It admits a basis given by the vectors $e_{\alphab}=e_{\alpha_1}\otimes ... \otimes e_{\alpha_{n(\alphab)}}$ where $\alphab\in \mathcal{I}(A)$, by convention, the space $T_0(V)$ is one-dimensional and generated by the vector $e_\emptyset$. We can endow $V$ with the scalar product $\langle.,.\rangle$ in which the vectors $(e_\alpha)_{\alpha}$ are orthonormal. We extend this to $T(V)$ by $\langle e_{\alphab},e_{\alphab'}\rangle =\frac{\delta_{\alphab,\alphab'}}{n(\alphab)!}$.\\
The group $\frak{S}_n$ acts on $T_n(V)$ by permuting variables, and this action is compatible with the one on multi-indices. The symmetric space $S_n(V)$ is the space of invariants. It admits a basis $e_\mu=\sum_{\alphab\in C(\mu)}e_{\alphab}$ indexed by partitions with $\mu\in \mathcal{P}_n(A)$. Let $s$ be the symmetrization operator on $T(V)$, we have the relation $s(e_{\alphab})=\mu_{\alphab}!~e_{\mu_{\alphab}}$. The  \textit{polynomial} Bosonic Fock space (or symmetric algebra) is defined as the direct sum
\begin{equation*}
S(V)=\bigoplus_n S_n(V)=\bigoplus_{\mu\in \mathcal{P}(A)} \Q~e_\mu.
\end{equation*}
The basis $(e_\mu)_{\mu\in \mathcal{P}(A)}$ satisfies the relation $\mu_1! \langle e_{\mu_1},e_{\mu_2}\rangle=\delta_{\mu_1,\mu_2}$.\\

\paragraph{Graduations and completions:}
Let $d: A\longrightarrow \N$ be a function with finite preimage, the completion $\hat{V}$ is $\prod_d V^d$ where $V^d$ is the space of homogeneous elements of degree $d$. The graduation admits a natural extension to multi-indices and partition by additivity
\begin{equation*}
d(\alphab)=\sum_i d(\alpha_i),~~~d(\mu)=\sum_\alpha \mu(\alpha)d(\alpha).
\end{equation*}
This induces a graduation on both $S(V)$ and $T(V)$ and a second graduation is given by the length $n$. The canonical basis are homogeneous for both $d$ and $n$. We denote $S_n^d(V)$ and $T_n^d(V)$ the homogeneous elements of degree $(d,n)$. The completion of $T(V)$ and $S(V)$ for the topology given by the two graduations $n,d$ will be denoted $\hat{T}(V)$ and $\hat{S}(V)$, we call the later space the \textit{formal} Bosonic Fock space and is given by
\begin{equation*}
\hat{S}(V)=\prod_{d,n} S_n^d(V).
\end{equation*}
Which can be identified with $\Q^{\mathcal{P}(A)}$, the elements of $\hat{S}(V)$ are formal sums $\sum_{\mu \in \mathcal{P}(A)} a(\mu)e_\mu.$

\begin{rem}
\label{remark_grad_operator}
The graduations $d,n$ are associated to two operators $D,N$ on $\hat{S}(V)$, they commute and are diagonal in the canonical basis $(e_\mu)_{\mu}$. We have $Ne_{\mu}=n(\mu) e_\mu~~\text{and}~~De_\mu = d(\mu) e_\mu.$
\end{rem}

\paragraph{Symmetric product (disjoint union):}
\label{paragraph_union}
The tensor product endows $T(V)$ with a structure of graded algebra:
\begin{equation*}
T_{n_1}^{d_1}(V)\otimes T_{n_2}^{d_2}(V)\longrightarrow T_{n_1+n_2}^{d_1+d_2}(V).
\end{equation*}
Applyed to the canonical basis it satisfies $e_{\alphab_1}\otimes e_{\alphab_2}=e_{\alphab_1\otimes \alphab_2}$. The symmetrization of the tensor product defines the symmetric product (disjoint union). It is defined by
\begin{equation*}
 x \sqcup y=\frac{s_{n_1+n_2}(x \otimes y)}{n_1!~n_2!},~~~~\forall (x,y)\in S_{n_1}(V)\times S_{n_2}(V).
\end{equation*}
Both $\otimes$ and $\sqcup$ are additive for the two graduations, then they are continuous for the topology and well defined on the completion. In the canonical basis$(e_\mu)_{\mu \in \Pc(A)}$, the product is characterized by
\begin{equation}
\label{formula_disjoint_union_vector}
e_\mu \sqcup e_\nu= \binom{\mu+\nu}{\mu~~\nu}e_{\nu+\mu},~~~\text{with}~~~\binom{\mu+\nu}{\mu~~\nu}=\frac{(\mu+\nu)!}{\mu!~\nu!}=\prod_\alpha \frac{(\mu(\alpha)+\nu(\alpha))!}{\mu(\alpha)!\nu(\alpha)!}.
\end{equation}
From Formula \ref{formula_disjoint_union_vector}, $(\hat{S}(V),\sqcup,+)$ is a commutative algebra and the unit is given by the vacuum $e_{\emptyset}$. If two partitions have disjoint supports, according to Formula \ref{formula_disjoint_union_vector}, they satisfy $e_\mu \sqcup e_\nu = e_{\mu+\nu}$, and we also have $e_{\alpha}^{\sqcup k} =e_{(\alpha)}^{\sqcup k}=k!e_{k(\alpha)}$ for all $\alpha \in A$ and $k\in \N$. From these two relations, we obtain:
\begin{equation}
\label{formula_free_product_S(V)}
\mu! e_\mu =\bigsqcup_\alpha e_\alpha^{\mu(\alpha)}.
\end{equation}
Then, it is natural to associate to a vector $a\in S(V)$ a generating series in the following way
\begin{equation*}
Z_a(\tb)=\sum_\mu \frac{a(\mu)\tb^\mu}{\mu!},~~~\text{where}~~~\tb^\mu=\prod_{\alpha\in A} t_\alpha^{\mu(\alpha)}.
\end{equation*}
The series $Z_a(\tb)$ belongs to $\C[\tb]$, the space of polynomials in $\tb=(t_a)_{a\in A}$ that depend only on a finite number of variables. According to Formula \ref{formula_disjoint_union_vector}, we can see that the application $a\to Z_a$ is multiplicative
\begin{equation*}
Z_{a_1\sqcup a_2}(\tb)=Z_{a_1}(\tb)Z_{a_2}(\tb).
\end{equation*}
Then $Z$ defines an isomorphism of algebras $Z:~S(V)\longrightarrow \Q[\tb]$, which preserves the graduations. Then, we can extend $Z$ to the completion $\hat{S}(V)$ and this induces an isomorphism:
\begin{equation*}
Z: \hat{S}(V)\longrightarrow \C[[\tb]].
\end{equation*}

\subsection{Operators on Fock spaces}
\label{subsection_operator_boson}
\paragraph{Operators:}
\label{paragraph_operator_boson}
Let $K\in \End(T(V))$ we denote  by $(K[\alphab^+|\alphab^-])_{(\alpha^+,\alpha^-)\in \mathcal{I}(A)^2 }$ the matrix coefficients of $K$ in the canonical basis.  To be well defined as an operator on $T(V)$, each column of the matrix must contain only a finite number of non-vanishing coefficients. If $K\in \End(S(V))$ it is representend by a matrix indexed by pairs of partitions $(K[\mu^+|\mu^-])_{(\mu^+,\mu^-)}$.\\
We mainly consider operators on the formal completions $\hat{T}(V)$ and $\hat{S}(V)$; in this case we have $K\in \End(\hat{T}(V))$ (or $\End(\hat{S}(V))$) iff the rows of its matrix are finite. \footnote{This is natural because the dual of the space $T(V)$ is identified with its formal completion $\hat{T}(V)$, the transpose of a matrix with finite columns has finite rows.} An operator is homogeneous of degree $(d,n)$ iff
\begin{equation*}
[D,K]=dK,~~~\text{and}~~~[N,K]=nK,
\end{equation*}
and then $K$ induces an operator $K \in \Hom(T_{*}^{*}(V), T_{*+n}^{*+d}(V))$. An homogeneous operator of degree $(d,n)$ always defines an operator in $\End(\hat{T}(V))$\footnote{A sum of the form  $K=\sum_{n,d}K^{d}_{n}$ does not always converge, because $(d,n)$ could take negative values. In the case of operators in $\Hom(\hat{T}_{n^-}(V),\hat{T}_{n^+}(V))$ an expression of the form $
\sum_{d\ge k} A_{n^+,n^-}^d~~~k\in \Z$ is always well defined.}.

\paragraph{Coproduct and union of operators:}
\label{paragraph_union_operator}
The product defines a map $\theta : S(V)\otimes S(V)\longrightarrow S(V)$, the coproduct is the dual operator
\begin{equation*}
\theta^* : S(V)^* \longrightarrow S(V)^*\otimes S(V)^*.
\end{equation*}
 Using the scalar product, for each $d$, $\theta^{*}$ defines a map
\begin{equation*}
\theta^* : S^d(V) \longrightarrow (S(V)\otimes S(V))^d.
\end{equation*}
The target is the subspace of elements of degree $d$ in the tensor product. We have the following formula:
\begin{equation}
\label{formula_coproduct_vector}
\theta^*e_\mu = \sum_{\mu_1+\mu_2=\mu} e_{\mu_1} \otimes e_{\mu_2}.
\end{equation}
Which defines an operator on both $S(V)$ and $\hat{S}(V)$.\footnote{From the last formula, we can see $\theta^*$ as the splitting operator and $\theta$ as the union operator.}\\
With these two operations, it is natural to define the product of two operators $K_1,K_2$ by
\begin{equation*}
\label{def_product_operators}
K_1\sqcup K_2 = \theta \circ (K_1\otimes K_2) \circ \theta^*.
\end{equation*}
This definition leads to the quite natural formula
\begin{equation}
\label{formule_union_op_1}
(K_1\sqcup K_2) e_\mu = \sum_{\mu_1+\mu_2=\mu} (K_1 e_{\mu_1}) \sqcup (K_2 e_{\mu_2})~~~~~~~\forall~\mu~\in \mathcal{P}(A).
\end{equation}
In terms of the coefficients of the matrices, this equivalent to:
\begin{equation}
\label{formule_union_op_coeff}
(K_1\sqcup K_2)[\mu^+|\mu^-]=\sum_{\mu^+=\mu^+_1+\mu^+_2,\mu^-=\mu^-_1+\mu^-_2}\binom{\mu^+}{\mu_1^+~ \mu_2^+}K_1[\mu^+_1|\mu^-_1]K_2[\mu^+_2|\mu^-_2].
\end{equation}

From Formulas \ref{formule_union_op_1} and \ref{formule_union_op_coeff}, we see that the union is well defined for operators on $S(V)$ and $\hat{S}(V)$, the sums that are involved are always finite. We can summarize this discussion by the following:

\begin{prop}
The disjoint union defines a commutative product:
\begin{equation*}
\End(\hat{S}(V))\otimes \End(\hat{S}(V))\longrightarrow \End(\hat{S}(V)).
\end{equation*}
Which preserves the two graduations.
\end{prop}
Let $k\in \N$ and $\pr_k$ be the projection
\begin{equation*}
    \pr_k: S(V)\longrightarrow S_k(V).
\end{equation*}
\begin{lem}
\label{lemma_identity_pr}
\begin{itemize}
    \item The projection $\pr_0$ satifies  $\pr_0\sqcup K=K$ for all $K\in \End(\hat{S}(V))$ then it is the unit for $\sqcup$.
    \item The power $\pr_1^{\sqcup k}$ is the symmetrization on $T_k(V)$, on $S(V)$, it satisfies $\frac{\pr_1^{\sqcup k}}{k!}=\pr_k$.
    \item We have the identity $\id = \exp_{\sqcup}(\pr_1)$ on $S(V)$.
\end{itemize}
\end{lem}

\begin{rem}
We can associate to an operator $K$ the formal series
\begin{equation*}
Z_K=\sum_{\mu_+,\mu_-}K[\mu_+|\mu_-]\frac{\tb_+^{\mu_+} \tb_-^{\mu_-}}{\mu_+!}.
\end{equation*}
It is an element of $\Q[\tb_-][[\tb_+]]$ if $K\in \End(\hat{S}(V))$ (resp. $\Q[\tb_+][[\tb_-]]$ if $K\in \End(S(V))$). Thus, similarly to the case of $\hat{S}(V)$, by using Formula \ref{formule_union_op_1}, these series satisfy $Z_{K_1\sqcup K_2}=Z_{K_1}Z_{K_2}$.
Indeed, the elements $t_{\alpha,+}$ and $t_{\alpha,-}$ correspond, respectively, to $e_{\alpha,+}=e_{\alpha}\otimes e_{\emptyset}^*$ and $e_{\alpha,-}=e_{\emptyset}\otimes e_{\alpha}^*$. If $E_{\mu^+,\mu^-}=e_{\mu^+}\otimes e_{\mu^-}^*$,  we have
\begin{equation*}
E_{\mu^+_1,\mu^-_1}\sqcup E_{\mu^+_2,\mu^-_2}=\binom{\mu^+_1 + \mu^+_2}{\mu^+_1~\mu^+_2}E_{\mu^+_1+\mu^+_2,\mu^-_1+\mu^-_2},
\end{equation*}
and we can deduce the factorization $ E_{\mu^+,\mu^-}= E_{\mu^+,\emptyset}\sqcup E_{\emptyset,\mu^-}.$ Then, the $E_{\mu^+,\mu^-}$ span a subalgebra of $\End(S(V))$, which is isomorphic to $S(V\oplus V^*)= S(V)\otimes S(V^*)$, which justifies the identification with $\Q[[\tb_+,\tb_-]]$.
\end{rem}

\paragraph{Creation and annihilation operator:}
\label{paragraph_creation}
In the theory of Fock spaces, for $\alpha \in A$, it is natural to consider the \textit{creation} operators
\begin{equation*}
\varphi_\alpha(x)=e_\alpha \sqcup x,~~~\text{then}~~~\varphi_\alpha e_\mu = (\mu(\alpha)+1)e_{\mu+(\alpha)}.
\end{equation*}
The \textit{annihilation} operator $\varphi_\alpha^*$ is defined by duality; we have
\begin{equation*}
\varphi_\alpha^* e_\mu^* = \mu(\alpha) e_{\mu-(\alpha)}^*,~~~\text{then}~~~\varphi_\alpha^* e_\mu=\delta_{\mu(\alpha)>0}e_{\mu-(\alpha)}.
\end{equation*}
Creation and annihilation operators belong to $\End(\hat{S}(V))\cap \End(S(V))$. The two operators $\varphi_\alpha,\varphi_{\alpha'}^*$ satisfy the usual commutation relation
\begin{equation*}
[\varphi_\alpha^*,\varphi_{\alpha'}]=\delta_{\alpha,\alpha'}.
\end{equation*}
By using the identification $S(V)\to \Q[t]$, it is straightforward to see that the operators $\varphi_\alpha,\varphi_\alpha^*$ are given by the following differential operators:
\begin{equation}
\label{formula_creation_annihilation_diff_op}
\varphi_\alpha=t_\alpha, ~~~\text{and}~~~\varphi_\alpha^*=\partial_\alpha,~~~\text{with}~~~\partial_\alpha=\frac{\partial }{\partial t_\alpha}.
\end{equation}
In other words, for all $\alpha\in A$
\begin{equation*}
Z_{\varphi_\alpha (a)}=t_\alpha Z_a,~~~\text{and}~~~ Z_{\varphi_\alpha (a)}=\partial_\alpha Z_a,~~~\forall~a\in \hat{S}(V)
\end{equation*}
 Let $\varphi_\mu=\prod_\alpha \varphi_\alpha^{\mu(\alpha)}$ and $\varphi_\mu^*$ the dual operator, we have:
\begin{equation*}
Z_{\varphi_\mu(a)}=\tb^{\mu}Z_x,~~~\text{and}~~~Z_{\varphi_\mu^*(a)}=\partial_\mu Z_a,
\end{equation*}
where we use the notation $\partial_\mu=\prod_\alpha\partial_\alpha^{\mu(\alpha)}$. The degree of $\varphi_{\mu^+} \varphi_{\mu^-}^*$ is $d(\mu^+)-d(\mu^-)$, then, series of the form:
\begin{equation*}
\sum_{\mu^+,\mu^-} D[\mu^+,\mu^-] \varphi_{\mu^+} \varphi_{\mu^-}^*,
\end{equation*}
are well defined on $S(V)$ iff the matrix $D[.,.]$ has finite columns and on $\hat{S}(V)$ iff it has finite rows. We denote $\D(V)$ and $\Dh(V)$ these subspaces of $\End(S(V))$ and $\End(\hat(S)(V)$.

\paragraph{Projective limit and differential operators:}
\label{paragraph_limit}
\label{paragraph_operator_proj_lim}
In what follows, there is a zero element $0\in A$ such that $d(0)=0$. We assume that such an element is unique and denote $A^*=A\backslash\{0\}$. Then we can consider the annihilation operator $\varphi_0^*$; it induces a family of operators $\varphi_0^*: S_{n+1}(V)\to S_{n}(V)$, that define a projective system $S_0(V)\longleftarrow S_1(V)\longleftarrow S_2(V) \longleftarrow \cdots$. We denote
\begin{equation}
\label{formula_def_Sinfty(V)}
S_\infty^d(V)=\varprojlim_n S_n^d(V),~~~\text{and}~~~\hat{S}_{\infty}(V)= \prod_d S_\infty^d(V).
\end{equation}
 The space $\hat{S}_{\infty}(V)$ is a subspace of $\hat{S}(V)$, a vector $a$ belong to $\hat{S}_{\infty}(V)$ iff it satisfies
\begin{equation*}
\varphi_0^* a=a.
\end{equation*}
In terms of formal series in $\Q[[\tb]]$, if $a\in \hat{S}_{\infty}(V)$, the formal series $Z_a$ satisfies $\partial_0Z_a=Z_a.$ Hence, $Z_a$ is of the form $Z_a(\tb)=\exp(t_0)Z_a(0,\tb^*)$, where we denote $\tb^*=(t_\alpha)_{\alpha\in A^*}$. In particular, for  $\mu\in \Pc(A^*)$, the vectors $\eb_{\mu}=e_\mu\sqcup \exp_{\sqcup}(e_0)$ form a basis of $S_{\infty}(V)$. They correspond to series $\tb^\mu \exp(t_0)$. In particular, the vacuum is
\begin{equation*}
 \eb_{\emptyset}=\exp_{\sqcup}(e_0),
\end{equation*}
and corresponds to the series $\exp(t_0)$.

\begin{rem}[Symmetric functions]
If $V=\Q[L]$ is the space of polynomials, $e_0$ corresponds to $1$, the constant polynomial; the map $\varphi_0^*$ corresponds to the evaluation of the first variable at $0$, i.e. $\varphi_0^*P(L)=P(0,L)$. The space $S_\infty(V)$ is, in this case, the space of symmetric functions with an arbitrary number of variables; $\hat{S}_\infty(V)$ is then the formal completion of this space for the graduation given by the degree.
\end{rem}

Let $K$ be an operator on $\hat{S}(V)$, we consider the operator $K\sqcup \id$.

\begin{prop}
\label{prop_unstable_diff_operators}
Let $K\in \End(\hat{S}(V))$ and $(K[\mu^+,\mu^-])$ be its matrix. We have the following expression in terms of creation and annihilation operators:
\begin{equation*}
K\sqcup \id = \sum_{\mu^+,\mu^-} A[\mu^+|\mu^-] \frac{\varphi_{\mu^+}\varphi_{\mu^-}^*}{\mu^+!}.
\end{equation*}
 \end{prop}
Then, in light of the isomorphism $\hat{S}(V)\simeq \Q[[\tb]]$, if $K\in \End(\hat{S}(V)$ see that $K\sqcup \id$ is a formal differential operator in $\Dh(V)$ :
\begin{equation}
\label{formula_Au_differential_op}
K\sqcup \id = \sum A[\mu^+|\mu^-] \frac{\tb^{\mu^+}\partial_{\mu^-}}{\mu^+!}.
\end{equation}
 The correspondence $K\rightarrow K\sqcup \id$ identifies $\End(\hat{S}(V))$ and $\Dh(V)$ (and also $\D(V)$ with $\End(S(V))$).
\begin{proof}
By using the Formula \ref{formule_union_op_1}, we obtain
\begin{equation*}
(K\sqcup \id) e_\mu= \sum_{\mu_1+\mu_2=\mu} (K e_{\mu_1})\sqcup e_{\mu_2}=\sum_{\mu^-,~\mu^+} \delta_{\mu\ge \mu^-}~K[\mu^+|\mu-]~e_{\mu^+}\sqcup e_{\mu-\mu^-}.
 \end{equation*}
We have for all $\mu^+,\mu^-,\mu$
 \begin{equation*}
\varphi_{\mu^+}e_\mu=\mu^+!~e_{\mu^+}\sqcup e_\mu,~~~\varphi_{\mu^-}^*e_\mu=\delta_{\mu\ge \mu^-}~e_{\mu-\mu^-},
\end{equation*}
and then we can rewrite
\begin{equation*}
e_{\mu^+}\sqcup e_{\mu-\mu^-} =\frac{\varphi_{\mu^+}e_{\mu-\mu^-}}{\mu^+!}=\frac{\varphi_{\mu^+}\varphi_{\mu^-}^*e_{\mu}}{\mu^+!}.
\end{equation*}
It allows us to prove Proposition \ref{prop_unstable_diff_operators}.
\end{proof}

\begin{rem}
We can see that
\begin{equation*}
E_{\mu^+,\emptyset}\sqcup \id =\frac{\varphi_{\mu^+}}{\mu^+!},~~~\text{and}~~~E_{\emptyset,\mu^-}\sqcup id=\varphi_{\mu^-}^*.
\end{equation*}
Then we have $E_{\mu^+,\mu^-}\sqcup \id=\frac{\varphi_{\mu^+}\varphi_{\mu^-}^*}{\mu^+!}$, which also allows us to recover the last formula.
\end{rem}
A direct computation gives the following proposition:
\begin{prop}
\label{}
There is a unique operator $K_1\ast K_2\in \End(\hat{S}(V))$ that satisfies
\begin{equation*}
(K_1\sqcup \id )\circ (K_2\sqcup \id ) = (K_1\ast K_2) \sqcup \id.
\end{equation*}
\end{prop}
\begin{proof}
This is a consequence of Proposition \ref{prop_unstable_diff_operators} and the fact that $\Dh(V)$ is stable under composition.
\end{proof}
\begin{rem}[Coefficients]
In terms of matrices, by using Proposition \ref{prop_unstable_diff_operators}, the expression of $\ast$ is encoded in the product of differential operators.
\end{rem}

The spaces $\D(V),\Dh(V)$ are not stable under $\sqcup$. The time-ordered product $::$ is defined on differential operators by:
\begin{equation*}
:\tb^{\mu_1^+}\partial_{\mu_1^-}\tb^{\mu_2^+}\partial_{\mu_2^-}:~=~\tb^{\mu_1^++\mu_2^+}\partial_{\mu_1^-+\mu_2^-}.
\end{equation*}
We have the following formula:
\begin{equation*}
(A_1\sqcup A_2)\sqcup \id=~:(K_1\sqcup id)(K_2\sqcup \id):.
\end{equation*}
\begin{rem}
Finally, if we assume that $\varphi_0^* K=0$. We can see that $K\sqcup id$  defines an operator on the projective limit $\hat{S}_{\infty}(V)$. Moreover, if $K_1,K_2$ satisfy the last assumption, then we have $\varphi_0^*\circ( K_1\ast K_2)=0.$ 
\end{rem}

\subsection{Space of directed surfaces and acyclic stable graphs}
\label{subsection_alg_dir_surf}
\paragraph{Symmetric algebra and directed surfaces:}
\label{paragraph_alg_dir_surface}
Let $\M^{\s}$ be the vector space over $\Q$ generated by isomorphism classes $\Mol$ of directed stable surfaces with labeled boundaries. We consider $\M$ the vector space generated by surfaces that can have connected components isomorphic to unstable cylinders of type $(0,1,1)$. By convention, we also introduce the empty surface $\emptyset$. For each $\Mol$, we denote $e_{\Mol}\in \M$ the corresponding vector, $\M$ contains the subspace $\M^{\cf}$ generated by connected surfaces, and $\M^{ \cf\s}=\M^{\cf}\cap \M^{\s}$.
There are three natural graduations on $\M$; the first is the opposite of the Euler characteristic denoted by $d$. The two other graduations $n^+,n^-$ are the numbers of positive and negative boundary components. For each $d,n^+,n^-$, the space $\M^{d}_{n^+,n^-}$ is finite dimensional, we can consider the completion $\hat{\M}$ of $\M$ for these three graduations. The group $\mathfrak{S}_{n^+,n^-}=\mathfrak{S}_{n^+}\times \mathfrak{S}_{n^-}$ acts on $\M_{n^+,n^-}$ by permuting the labels. As we see in Remark \ref{remark_label_1}, a directed labeled surface defines a pair of partitions of the positive and negative boundary components $(I^+,I^-)$, and the group $\mathfrak{S}_{n^+,n^-}$ acts on them. We denote $S_{n^+,n^-}(\M)$ the subspace of invariant elements under this action, and we define the symmetric space:
\begin{equation*}
S(\M)=\bigoplus_{n^+,n^-} S_{n^+,n^-}(\M).
\end{equation*}
It corresponds to the free module generated by surfaces with unlabeled boundaries (i.e., elements in $\bordo$). If $\Mo\in \bordo$, we denote $e_{\Mo}$ as the corresponding element of $S(\M)$, which is the sum of all possible ways $\Mol$ to label the boundary components of $\Mo$. Connected surfaces are fixed by the action of the symmetric group, we have $\M^{\cf}\subset S(\M)$, for a surface of type $(g,n^+,n^-)$ we can denote $e_{g,n^+,n^-}$ the corresponding vector. It is still possible to consider the completion of $S(\M)$ for the graduations $d,n^+,n^-$, we denote it $\hat{S}(\M)$.
\begin{rem}[Time inversion]
There is a natural involution that acts by reversing the sign of the boundaries
\begin{equation*}
\iota: \M_{n^+,n^-}\longrightarrow \M_{n^-,n^+},
\end{equation*}
it preserves the Euler characteristic. This induces an involution on $\M,S(\M),\hat{S}(\M)...$
\end{rem}
\paragraph{Disjoint union:}
\label{paragraph_union_surface}
There is a product on $\M$ given by the disjoint union of directed surfaces. Let ${\Mol}_1,{\Mol}_2\in \bordol$ be two directed surfaces, by abuse of notations, we denote ${\Mol}_1\otimes {\Mol}_2 \in \bordol$ the surface obtained by taking the disjoint union and enumerating the boundaries by starting with the ones of ${\Mol}_1$. We define:
\begin{equation*}
e_{{\Mol}_1}\sqcup e_{{\Mol}_2}=\frac{1}{n_1^+!n_1^-!n_2^+!n_2^-!}\sum_{\sigma=(\sigma^+,\sigma^-)}\sigma\cdot (e_{{\Mol}_1\otimes {\Mol}_2}),
\end{equation*}
 we sum over the group $\mathfrak{S}_{n^+_1+n_2^+,n^-_1+n_2^-}$. As in the case of symmetric algebra, the disjoint union  $S(\M)\otimes S(\M) \longrightarrow S(\M)$ is associative, commutative and linear; then $(S(\M),+,\sqcup)$ is a commutative algebra. It is also additive for the graduations and then extends to the completion. The space $S_{0,0}(\M)$ is one dimensional, the generator is the empty surface $e_\emptyset$ and corresponds to the unit for $\sqcup$. The disjoint union $\sqcup$ also defines a monoidal structure on $\bordot$ moreover, it is compatible with the union on $S(\M)$. If for all $\Mo\in \bordo$ and for all $\Mo_1\in \bordoct$ we denote $\mu_{\Mo}(\Mo_1)$ the number of connected components of $\Mo$ isomorphic to $\Mo_1$, this defines an element in $\mathcal{P}(\bordoct)$. For all $\Mo_1,\Mo_2\in \bordot$, we have the analog of Formula \ref{formula_disjoint_union_vector}
\begin{equation*}
e_{\Mo_1}\sqcup e_{\Mo_2} = \binom{\nu_{\Mo_1}+\nu_{\Mo_2}}{\nu_{\Mo_1}~\nu_{\Mo_2}}e_{\Mo_1\sqcup \Mo_2}.
\end{equation*}
Then, similarly to the case of the Fock space, the algebra $(S(\M),\sqcup,+)$ is the free commutative algebra generated by the connected surfaces in $\bordoct$. In particular, we have the following formula that relates connected and disconnected surfaces:
\begin{equation}
\label{formula_exp_surface_sqcup}
 \sum_{\Mo\in \bordot}e_{\Mo}=\exp_{\sqcup}\left(\sum_{\Mo\in \bordoc}e_{\Mo}\right).
\end{equation}

\paragraph{Acyclic stable graphs:}
\label{paragraph_algebra_acyclic}
In some sense we want to define composition in $\hat{S}(\M)$ that should be related to composition of operators. To do that we consider the vector space $\A$ generated by elements $\Go$ in $\acyclolp$ the set of isomorphism class of labeled acyclic stable graph with a finite number of components isomorphic to cylinders (see Figure \ref{figure_compo_stable_graph}). As before, the symmetric group acts on $\A$, and we denote $S(\A)$ the subspace of symmetric elements. It is the vector space generated by the vectors $e_{\Go}$ for $\Go \in \acyclop$. The disjoint union $\sqcup$ is also well defined, and $S(\A)$ is the free commutative algebra generated by the connected acyclic directed stable graphs and the cylinder. \\
The composition ${\Gol}_1\cdot {\Gol}_2$ of two directed stable graphs is defined by gluing the negative boundary components of ${\Gol}_1$ to the positive components of ${\Gol}_2$. We forget the curves in this new graph that are boundaries of cylinders (see Figure \ref{figure_compo_stable_graph}). As we see in \cite{barazer2021cuttingorientableribbongraphs}, the result is always an acyclic stable graph, this define a binary operation:
\begin{equation*}
\acyclolp \times \acyclolp \longrightarrow \acyclolp.
\end{equation*}
This operation is associative and defines a structure of algebra on $\A$ by
\begin{equation*}
 e_{{\Gol}_1}\cdot e_{{\Gol}_2}=\frac{e_{{\Gol}_1\cdot {\Gol}_2}}{k!},
\end{equation*}
where $k=n^-({\Gol}_1)=n^+({\Gol}_2)$ is the number of boundaries glued. Contrary to the union, the composition is not commutative; moreover, it is not distributive for $\sqcup$, and then $(S(\A),\cdot,\sqcup)$ is not a ring. But $S(\A)$ admit two structures of algebra.
\begin{figure}
\centering
\includegraphics[height=3cm]{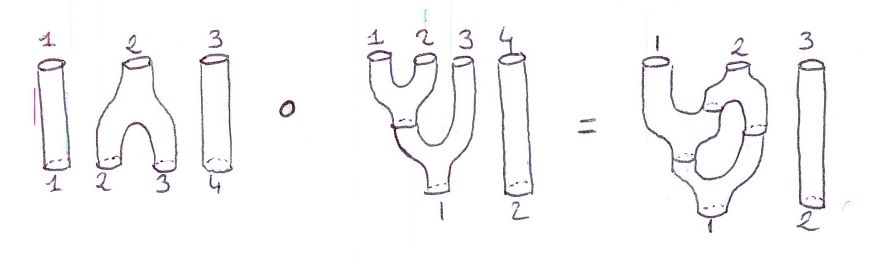}
\caption{Composition of two acyclic stable graphs.}
\label{figure_compo_stable_graph}
\end{figure}

\paragraph{Projective limit:}
\label{paragraph_proj_lim_cyl}
We can forget the number of cylinders by using projective limit, as in the case of Fock space. Let $\varphi_{0,1,1}^*:S_{n^++1,n^-+1}(\A)\to S_{n^+,n^-}(\A)$ be the annihilation operator that removes a cylinder $e_{0,1,1}$, defined by analogy to the Fock space. In this way we obtain a projective system and we denote $S_\infty(\A)$ the projective limit. As a vector space, $S_{\infty}(\A)$ is isomorphic to the space $S(\A^{\s})$ generated by elements in $\acyclo$. If $\Go\in \acyclo$, it is represented by the vector
\begin{equation*}
\eb_{\Go}=e_{\Go}\sqcup \exp_{\sqcup}(e_{0,1,1}).
\end{equation*}
The interpretation is the following: we add to a graph an arbitrary number of cylinders, and as we see later, this allows more possibilities of gluing's when we compose stable graphs. As in the case of operators on Fock spaces, we have the following proposition:
\begin{prop}
For all $\Go_1,\Go_2\in \acyclo$, there is a unique $e_{\Go_1}\ast e_{\Go_2}$ in $S(\A^\s)$ such that:
\begin{equation*}
(e_{\Go_1}\ast e_{\Go_1}) \sqcup \exp_{\sqcup}(e_{0,1,1})=\eb_{\Go_1} \circ \eb_{\Go_2}.
 \end{equation*}
\end{prop}
Then the composition induces a product structure on $S_{\infty}(\A)$. The result of the composition of two graphs is given in Figure \ref{figure_compo_A_infty}, to contrast with the usual composition. If we add an arbitrary number of cylinders, we allow all possible gluing.
\begin{figure}
\centering
\includegraphics[width=15cm]{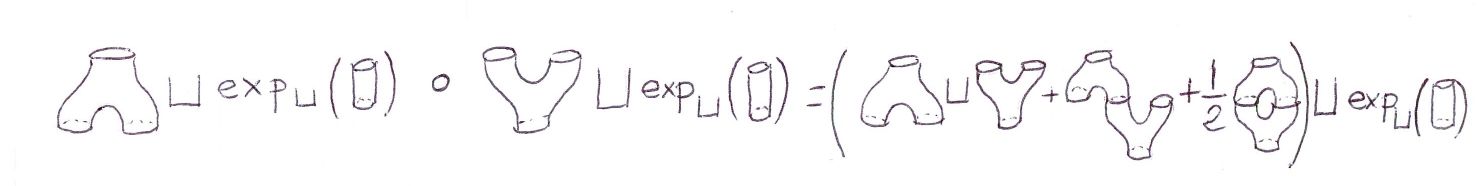}
\caption{Composition of two graphs in $S_{\infty}(\A)$.}
\label{figure_compo_A_infty}
\end{figure}

\paragraph{Exponential structure for the generating series of acyclic graphs:}
 \label{paragraph_exponential_structure}
We give some formulas relating the exponential of elements in $S_{\infty}(\A)$ for the composition to generating series of acyclic stable graphs. For each $\Go$, we denote $n_{\Go}$ as the number of linear orders on $\Go$.
\begin{prop}
\label{prop_exp_A}
In $\hat{S}_\infty(\A)$, we have the following formula:
\begin{equation}
\label{formula_exp_general}
\exp\left( \sum_{g,n^+,n^-} \eb_{g,n^+,n^-}q^{2g-2+n^++n^-}\right)=\sum_{\Go\in\acyclo}\frac{n_{\Go}q^{d(\Go)}}{\#X_0\Go!~\#\Aut(\Go)}\eb_{\Go},
\end{equation}
The sum in the $LHS$ run over the set of connected stable surfaces and sums on the $RHS$ runs over the set of acyclic stable graphs, possibly empty.
\end{prop}

The exponential means that we consider all the ways to compose directed surfaces together. It is then quite natural that we obtain, in this way, all the possible acyclic directed stable graphs; the composition of surfaces cannot create cycles in the graph. The proof of these formulas uses the fact that directed graphs are encoded by a family of surfaces with permutations to define the gluings. We don't give the details here. In particular, the exponential of
\begin{equation*}
P=\eb_{0,1,2}+\eb_{0,2,1},
\end{equation*}
corresponds to acyclic pants decompositions:
\begin{equation}
\label{formula_exp_pants}
\exp(qP)=\sum_{\Go\in\acyclogen}\frac{n_{\Go}q^{d(\Go)}}{d(\Go)!~\#\Aut(\Go)}\eb_{\Go}.
\end{equation}
Where $\acyclogen$ is the set of directed acyclic stable graphs such that each component is either of type $(0,1,2)$ or $(0,2,1)$. We illustrate in Figure \ref{figure_exp_A} the first terms of this formula.
\begin{figure}
 \centering
\includegraphics[width=12cm]{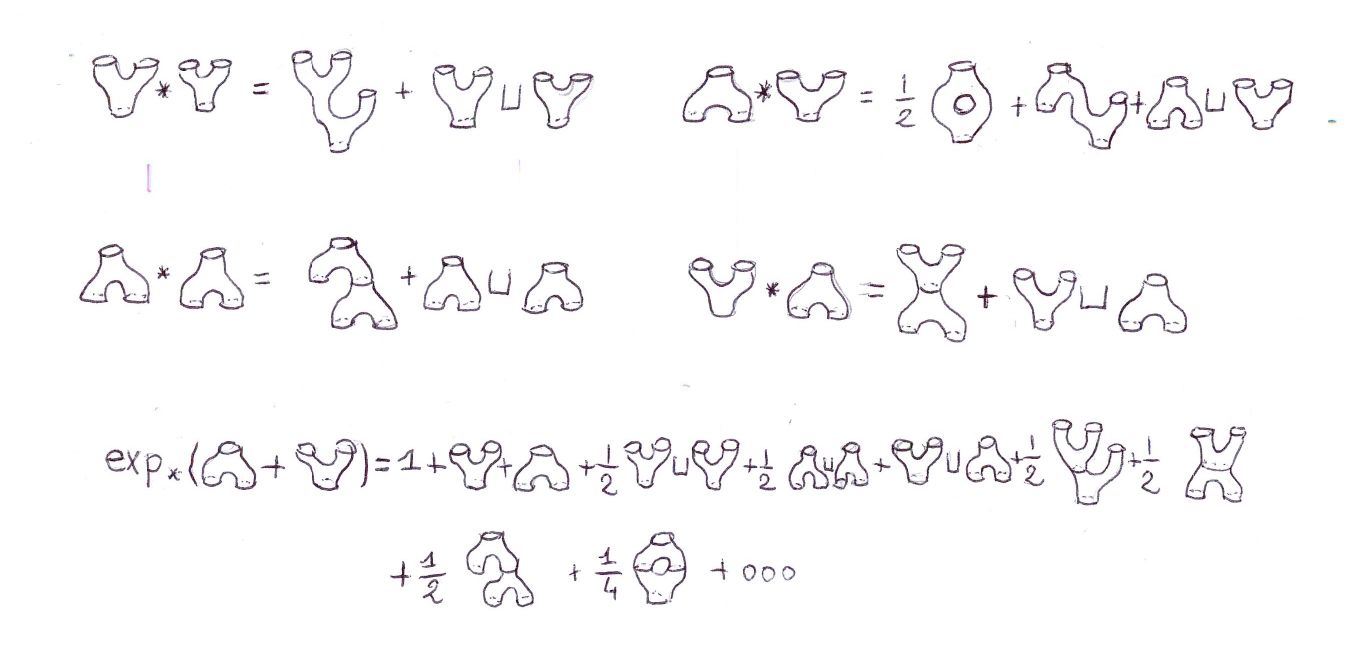}
\caption{First terms of the expansion of $\exp(qP)$ in the variable $q$.}
\label{figure_exp_A}
\end{figure}

\paragraph{Representations:}
\label{paragraph_representation}
Let $V$ as in Subsection \ref{subsection_operator_boson}, we consider a family of linear operators
\begin{equation*}
K_{g,n^+,n^-} : \hat{S}_{n^-}(V) \to \hat{S}_{n^+}(V).
\end{equation*}
We assume that $K_{g,n^+,n^-}$ is homogeneous of degree $d_{g,n^+,n^-}$. As $S(\M)$ is the free algebra generated by connected surfaces, we can define a unique linear morphism
\begin{equation*}
\Kb: S(\M)\longrightarrow \End(\hat{S}(V)),
\end{equation*}
Which preserves the three graduations $(d,n^+,n^-)$, satisfies $\Kb(x\sqcup y)=\Kb(x)\sqcup \Kb(y)$ and $\Kb(e_{g,n^+,n^-})=K_{g,n^+,n^-} $.
If we assume that $K_{0,1,1}=\pr_1$ is the identity of $V$, then we can extend $K$ to $S_{\infty}(\M)$. We can ask if it is possible to obtain a representation of $S_{\infty}(A)$.
\begin{prop}
For each $K$ as bellow, there exists a unique linear map:
\begin{equation*}
K : S_{\infty}(\A)\longrightarrow \Dh(V),
\end{equation*}
that extends $K$ on $\M$, and satisfies: 
\begin{equation*}
K(e_{\Go_1}\sqcup e_{\Go_2})=K(e_{\Go_1})\sqcup K(e_{\Go_2}),~~~\text{and}~~~K(e_{\Go_1}\cdot e_{\Go_1})=K(e_{\Go_1})\circ K(e_{\Go_2})
\end{equation*}
\end{prop}
In particular, we have $\Kb(\exp_{\sqcup}(e_{0,1,1}))=\exp_{\sqcup}(K_{0,1,1})\exp_{\sqcup}(\pr_1)=\id$; then,for each $(g,n^+,n^-)$ the operator:
\begin{equation*}
    \Kb(\eb_{g,n^+,n^-})=K_{g,n^+,n^-}\sqcup \id,
\end{equation*}
is a differential operator.

\begin{proof}
We prove that we can extend the representation to a morphism
\begin{equation*}
\A\longrightarrow \End(\hat{S}(V)).
\end{equation*}
We can found an explicit formula for $K_{\Go}$. We fix $\Go$, for each $c\in X_0\Go$ there is a projection:
\begin{equation*}
L^\pm_c: A^{X_1\Go} \longrightarrow A^{n^\pm(c)}
\end{equation*}
We also have projection $L^\pm: A^{X_1\Go} \longrightarrow A^{\partial^{\pm}\Go}$.
For each $\alphab^{\pm}\in A^{\partial^{\pm}\Go}$ let
\begin{equation*}
\Lambda_{\Go,A}(\alphab^+,\alphab^-)=\{\alphab\in A^{X_1\Go}~|~L^{+}(\alphab)=\alphab^+~,~L^{-}(\alphab)=\alphab^-\}.
\end{equation*}
Then we set
\begin{equation*}
K_{\Go}[\alphab^+,\alphab^-]=\frac{1}{n-(\Go)!}\sum_{\alphab\in \Lambda_{\Go,A}(\alphab^+,\alphab^-)}\prod_c \frac{K_{\Go(c)}'[L^+_c(\alphab),L^-_c(\alphab)]}{n^-(\Go(c))!}.
\end{equation*}
According to the fact that the operators $K_{g,n^+,n^-}$ are symmetric and are defined on $\hat{S}(V)$,  we can see that the last sum is indeed finite. Moreover, we can see that for each $\alphab^{+}$, the number of $\alphab^-$ such that $ K_{\Go}[\alphab^+,\alphab^-]$ is non-zero is actually finite. This is due to the fact that this property is true for each $K_{g(c),n^+(c),n^-(c)}$ and $\Go$ is acyclic. Then $K_{\Go}$ is a matrix with finite rows and defines an operator in $\End(\hat{T}(V))$. By construction, we can see that
\begin{equation*}
K_{\Go_1\cdot \Go_2}[\alphab^+,\alphab^-]=\sum_\beta K_{\Go_1}[\alphab^+,\beta] K_{\Go_2}[\beta,\alphab^-].
\end{equation*}
And then we obtain a morphism for the composition. Now if $\Gol$ is a labeled graph with a linear order for each $i$, we can add a sufficient number of cylinders to $\Go(i)$ to obtain a new graph $\tilde{\Go}(i)$ and find a label $\tilde{\Gol}(i)$ of $\tilde{\Go}(i)$ to obtain the factorization:
\begin{equation*}
\Gol=\tilde{\Gol}(r)\cdot ...\cdot\tilde{\Gol}(1),
\end{equation*}
with $r=\#X_0\Go$. Then the algebra $\A$ is generated by $\M$, and this allows us to obtain the uniqueness.
\end{proof}
\section{Operators associated to directed ribbon graphs}
\label{section_operator_volumes}
In this part, we construct a family of integral operators $K_{g,n^+,n^-}$ with kernels $V_{g,n^+,n^-}$ and we study their elementary properties. We uses the notations of Section \ref{section_background} and set $V=\Q[L]$ the space of polynomials, these operators are defined on the formal Fock space $\hat{S}(V)$, which is identified with the space $\Q[[\tb]]$ of series with a countable number of variables $\tb=(t_0,t_1,\dots)$. We start by recalling some result about directed ribbon graphs and volumes of their moduli spaces, after we state Proposition \ref{lem_transfert_lemma} that ensure the consistency of the operators.

\subsection{Directed ribbon graph and volumes of moduli spaces}
\label{subesection_directed_ribbon_graphs}
\paragraph{Definition:}
We refer to \cite{barazer2021cuttingorientableribbongraphs} for a more precise definition; a ribbon graph $R$ is defined by a set of half edges $XR$ and a triple of permutations $(s_0,s_1,s_2)$. We denote $X_iR$, the $i-$simplex of $R$, which corresponds to the orbits of $s_i$. $X_0R$ is the set of vertices, $X_1R$ the set of edges and $X_2R$ are the boundary components (or faces). A direction on $R$ is a map $\epsilon : XR \to \{\pm 1\}$ that satisfies 
\begin{equation*}
    \epsilon \circ s_0=-\epsilon,~~~\text{and}~~~ \epsilon \circ s_1=-\epsilon.
\end{equation*}
A directed ribbon graph $\Ro=(R,\epsilon)$ is a ribbon graph with a choice of direction.\\

A ribbon graph defines a surface $M_R\in \bord$ that is its topological realization, we have the identification $\partial M_R=X_2R$. A direction on $R$ is invariant under $s_2$ and then induces a non-constant map 
\begin{equation*}
    \epsilon : X_2R \to \{\pm 1\}.
\end{equation*}
This defines a direction $\Mo_{\Ro}$ on $M_R$. A directed ribbon graph is characterized by the fact that two adjascent boundaries have opposite signs; in other word, a ribbon graph is directed iff its dual is bipartite.\\

A metric on a ribbon graph $R$ is a map from the set of edges to $\Rpp$; we denote by $\Met(R)$ the cone of metrics on $R$. If $m\in \Met(R)$ and $\beta \in \partial M_R$ is a boundary component, it is possible to compute the length $\beta$ with respect to $m$ by summing the length of the edges that $\beta$ contains. It defines a function $l_\beta : \Met(R)\to \Rpp$ and we denote $L_{\partial}=(l_\beta)_{\beta\in \partial M_R}$. If $\Ro$ is directed, the biparticity implies that 
\begin{equation*}
    \sum_\beta \epsilon(\beta) l_\beta(m)= 0,
\end{equation*}
then $L_{\partial}$ takes its values in $\Lambda_{\Mo_{\Ro}}$. Usually, the directed ribbon graphs that we consider have labeled boundary components, and we denote $\beta_i^+$ (resp. $\beta_i^-$) the $i-$th positive (resp. negative) boundary components. We will also use the notation $\alpha_i^{\pm}(\Ro)$ the number of edges contained in $\beta_i^\pm$; we can form $\alphab^{\pm}(\Ro)=(\alpha_1^{+}(\Ro),...)$, which is a multi-index. Using the biparticity we also have the relation 
\begin{equation*}
  d(\alphab^{+}(\Ro)) = \sum_i \alpha_i^{+}(\Ro)=\sum_i \alpha_i^{-}(\Ro)= d(\alphab^{-}(\Ro)).
\end{equation*}
Let $e$ be an edge, we also denote $[e]^{+}\in \llbracket 1,n^{+}\rrbracket$ (resp. $[e]^-$) the positive (resp. negative) boundary that contain $e$.
\paragraph{Moduli spaces:}
Let $\Mo$ be a directed surface we denote $\rib(\Mo)$ the isomorphism classes of directed ribbon graphs embedded in $\Mo$. It corresponds to the pairs $(\Ro,\phi)$ where $\phi: \Mo_{\Ro}\to \Mo$ is a homeomorphism that preserves the direction of the boundaries components; two pairs are considered to be equivalent up to the action of the group of homeomorphisms of $M$ that fix each boundary component. The combinatorial moduli space $\Mc(\Mo)$ is then 
\begin{equation*}
    \Mc(\Mo)= \sqcup_{\Ro} \Met(\Ro)/\Aut(\Ro).
\end{equation*}
It is a cell complex and the boundary length induces a map $\Lbord :  \Mc(\Mo) \to \Lambda_{\Mo},$. For each $L\in \Lambda_{\Mo}$ we denote $\Mc(\Mo,L)= \Lbord^{-1}(L)$ the level set. From the results of \cite{barazer2021cuttingorientableribbongraphs} we have the following theorem:
\begin{thm}
    There are natural Lebesgue measures $d\mu_{\Mo},d\sigma_{\Mo},d\mu_{\Mo}(L)$ on $\Mc(\Mo),\Lambda_{\Mo}$ and $\Mc(\Mo,L)$. Moreover, the measure $d\mu_{\Mo}(L)$ on $\Mc(\Mo,L)$ is the conditional measure of $d\mu_{\Mo}$ with respect to $d\sigma_{\Mo}$.  The volume $V_{\Mo}(L)$ of the fiber $\Mc(\Mo,L)$ is a homogeneous piece-wise polynomial of degree $2d(\Mo)-n^+(\Mo)-n^-(\Mo)+1$.
\end{thm}
We generally work with surfaces with labeled boundaries; if $\Mol$ is labeled of type, $(g,n^+,n^-)$ we prefer the notation $V_{g,n^+,n^-}$. In this case $\Lambda_{\Mo}\simeq \Lambda_{n^+,n^-}\subset \Rpp^{n^+}\times \Rpp^{n^-}$, then $V_{g,n^+,n^-}$ can be seen as a function of two variables $(L^+,L^-)$ and we use the notation $V_{g,n^+,n^-}(L^+|L^-)$.

\subsection{Transfer lemma}

We start by introducing some notations, for each $n\in \Npp$ let $\Delta_n$ be the standard simplex, for $t>0$ let $t\cdot \Delta_n$ the dilatation by $t$ and $d\sigma^t_n$ the   Lebesgue measure on $t\cdot \Delta_n$ normalized by $\frac{t^{n-1}}{(n-1)!}$. For a vector $L\in \Rp^n$ we also use $|L|=\sum_i L_i$. The goal of this section is to prove the following proposition:
\begin{prop}
\label{prop_transfert_lemma}
Let $(g,n^+,n^-)$ and $\alphab\in \N^{n^-}$ be a multi-index, then the integral:
\begin{equation*}
\prod_{i}L_i^+~\int_{L^-\in |L^+|\cdot\Delta_{n^-}}V_{g,n^+,n^-}(L^+|L^-)\prod_i (L_i^-)^{\alpha_i^-}d\sigma_{n^-}^{|L^+|},
\end{equation*}
is well defined for $L^+\in(\Rp)^{n^+}$ and is a homogeneous symmetric polynomial of degree $2d_{g,n^+,n^-}+d(\alphab)$.
\end{prop}

The proof of Proposition \ref{prop_transfert_lemma} is a consequence of Lemma \ref{lem_transfert_lemma} given below. Let $\Ro$ be a directed ribbon graph, we set:
\begin{equation}
\label{formula_PRo}
\overline{P}_{\Ro}(\alphab^+|\alphab^-)=\#\left\{x\in \N^{X_1R}~|~ L_i^\pm(x)=\alphab^\pm_i,~~~\forall i\in \llbracket 1,n^\pm\rrbracket \right\}.
\end{equation}
\begin{lem}
\label{lem_transfert_lemma}
Let $\Ro$ be a directed ribbon graph and $\alphab^-\in \N^{n^-}$ be a multi-index. We have the relation:
\begin{equation}
\label{formula_integral_VRo}
\prod_i L_i^+\int_{L^-\in |L^+|\cdot\Delta_{n^+}}V_{\Ro}(L^+|L^-)\prod_i \frac{(L_i^-)^{\alpha_i^-}}{\alpha_i^-!}d\sigma_{n^-}^{|L^+|}= \sum_{\alphab^+} \overline{P}_{\Ro}(\alphab^+|\alphab^-) \prod_i\frac{(L_i^{+})^{\alpha_i^++\alpha_i^+(\Ro)}}{(\alpha_i^++\alpha_i^+(\Ro)-1)!}.
\end{equation}
\end{lem}

\begin{proof}
The space $\Met(\Ro,L^+,L^-)$ is defined by the relations $L^\pm_i(x)=L^\pm_i$ for all $i\in \llfloor 1 , n^{\pm}\rrfloor$. Let $\Met_+(\Ro,L^+)$ be the space defined by $L^+_i(x)=L^+_i$ for all $i\in \llbracket 1 , n^{+}\rrbracket$. The Lebesgue measure $d\mu_{\Ro,+}(L^+)$ on $\Met_+(\Ro,L^+)$ satisfies:
\begin{equation*}
\int_{x\in \Met_+(\Ro,L^+)}\prod_i(L^-_i(x))^{\alpha^-_i}d\mu_{\Ro,+}(L^+)=\int_{L^-\in |L^+|\cdot \Delta_{n^-}}V_{\Ro}(L^+|L^-)\prod_i (L^-_i)^{\alpha^-_i}d\sigma_{n^-}^{|L^+|}~~~~~\forall L^+\in~\Rp^{n^+}.
\end{equation*}
Each edge $e$ belongs to a unique positive boundary component, and this forms a partition of the set of edges. Then, we need to compute the integral of $\prod_i L_i^-(x)^{\alpha_i^+}$ over the set
\begin{equation*}
\left\{x~|~L_i^+(x)=L_i^+~~\forall i \in \llfloor 1 , n^{+}\rrfloor\right\}=\prod_i\left\{(x_e)_{e,[e]^+=i}\in \R_+^{\alpha_i^+(R)}~|~\sum_{e,[e]^+=i} x_e=L_i^+\right\}.
\end{equation*}
This set is a product of simplicies; the decomposition preserves the lattice of integer points, and then the measure on the RHS is the product of the Lebesgue measures on each factor. Moreover, by the Newton binomial formula, we have
\begin{equation*}
\prod_i L_i^{-}(x)^{\alpha_j^-}=\prod_i \alpha_i^-! \sum_{m\in \N^{X_1R}, L_i^-(m)=\alpha_i^-} \prod_e \frac{x_e^{m_e}}{m_e!},
\end{equation*}
when we integrate over the product of simplicies, it leads to
\begin{equation*}
\int_{\{|L_i^+(x)=L_i^+\}} \prod_e \frac{x_e^{m_e}}{m_e!} dx=\prod_i \int_{(x_e)_{e,[e]^+=i},\sum_{e,[e]^+=i} x_e=L_i^+} \prod_{e,[e]^+=i}\frac{x_e^{m_e}}{m_e!}d\sigma_i=\prod_i \frac{(L_i^+)^{\alpha_i^++\alpha_i^+(\Ro)-1}}{(\alpha_i^++\alpha_i^+(\Ro)-1)!},
\end{equation*}
because of the formula
\begin{equation*}
\int_{(x_e)_{e,[e]^+=i},\sum_{e,[e]^+=i} x_e=L_i^+} \prod_{e,[e]^+=i}\frac{x_e^{m_e}}{m_e!}d\sigma_i= \frac{(L_i^+)^{\sum_e (m_e+1)-1}}{(\sum_e (m_e+1)-1)!}.
\end{equation*}
Then, we finaly obtain: 
\begin{equation*}
\int_{L^-\in |L^+|\cdot \Delta_{n^-}}V_{\Ro}(L^+|L^-)\prod_i(L^-_i)^{\alpha^-_i}d\sigma_{n^-}^{|L^+|}=\sum_{\alphab^+} \overline{P}_{\Ro}(\alphab^+|\alphab^-)\frac{\alphab^-!\prod_i (L^+_i)^{\alpha^+_i+\alpha^+_i(\Ro)-1}}{(\alphab^+ +\alphab^+(\Ro)-1)!}.
\end{equation*}
\end{proof}

\begin{proof} 
To prove Proposition \ref{prop_transfert_lemma}, we use the fact that $V_{g,n^+,n^-}$ is given by
\begin{equation*}
    V_{g,n^+,n^-}=\sum_{\Ro \in \rib_{g,n^+,n^-}^*} \frac{V_{\Ro}}{\#\Aut(\Ro)},
\end{equation*}
where the sum run over the set $\rib_{g,n^+,n^-}^*$ of isomorphism classes of quadrivalent ribbon graphs of type $(g,n^+,n^-)$. Then with Lemma \ref{lem_transfert_lemma} we obtain:
\begin{eqnarray*}
\prod_i L_i^+~ \underset{L^-\in |L^+|\cdot \Delta_{n^-}}{\int}V_{g,n^+,n^-}(L^+|L^-)\prod_i \frac{(L_i^-)^{\alpha_i^-}}{\alpha_i^-!}d\sigma_{n^-}^{|L^+|}&=&\sum_{\Ro\in \rib_{g,n^+,n^-}^*} \frac{1}{\#\Aut(\Ro)}\underset{L^-\in |L^+|\cdot \Delta_{n^-}}{\int}V_{\Ro}(L^+|L^-)\prod_i \frac{(L_i^-)^{\alpha_i^-}}{\alpha_i^-!}d\sigma_{n^-}^{|L^+|}\\
&=&\underset{\Ro\in \rib_{g,n^+,n^-}^*}{\sum}\sum_{\alphab^+}\frac{ \overline{P}_{\Ro}(\alphab^+|\alphab^-)}{\#\Aut(\Ro)}\prod_i\frac{(L_i^{+})^{\alpha_i^++\alpha_i^+(\Ro)}}{(\alpha_i^++\alpha_i^+(\Ro)-1)!}.
\end{eqnarray*}
We can see that if $ \overline{P}_{\Ro}(\alphab^+|\alphab^-)$ is non-zero, we must have
\begin{equation*}
d(\alphab^+)=d(\alphab^-).
\end{equation*}
And then, the degree of the monomials in the RHS is 
\begin{equation*}
d(\alphab^+(\Ro))+d(\alphab^-)=\# X_1\Ro+d(\alphab^-)=2d_{g,n^+,n^-}+d(\alphab^-).
\end{equation*}
Then, the RHS is an homogeneous polynomial of degree $d(\alphab^-)+ 2d_{g,n^+,n^-}$. As the function $V_{g,n^+,n^-}$ is symmetric, we conclude that the polynomial is symmetric.
\end{proof}

\subsection{Operators associated to directed surfaces}
\paragraph{Definition of the operators:}
Let $(g,n^+,n^-)$ be such that $2g-2+n^++n^->0$, and: 
\begin{equation*}
K_{g,n^+,n^-}(L^+|L^-)=\prod_i L_i^+ ~~ V_{g,n^+,n^-}(L^+|L^-).
\end{equation*}
For each $P$ homogeneous polynomial, we can consider the integral
\begin{equation}
 K_{g,n^+,n^-}\cdot P~(L^+)=\frac{1}{n^-!}\int_{L^-\in |L^+|\cdot\Delta_{n^-}} K_{g,n^+,n^-}(L^+|L^-)P(L^-)d\sigma_{n^-}^{|L^+|}.
\end{equation}
The factor $\frac{1}{n^-!}$ is to avoid over-counting when composing the operators.
According to Proposition \ref{prop_transfert_lemma}, the integral is well defined, and $ K_{g,n^+,n^-}\cdot P$ is an homogeneous symmetric polynomial of degree $2d_{g,n^+,n^-}+d(P)$.\\

Using the formalism of Section \ref{subsection_boson} with $V=\Q[L]$, the tensor power $T_n(V)=\Q[L_1,...,L_n]$ is the space of polynomials in $n$ variables. For all multi-index $\alphab$ we denote 
\begin{equation*}
    e_{\alphab}(L)=\prod_i \frac{L_i^{\alpha_i}}{\alpha_i!},
\end{equation*}
which form a basis of $T(V)$. In this space, the natural grading $d$ is given by the degree, with $d(e_{\alphab})=d(\alphab)$. The completion $\hat{T}_n(V)=\Q[[L_1,...,L_n]]$ is the space of formal series in $n$ variables, $S(V)$ is the space of symmetric polynomials and $\hat{S}(V)$ is its completion. Then, using Proposition \ref{prop_transfert_lemma}, we obtain the following corollary:
\begin{cor}
\label{cor_operator_gn+n-}
The $K_{g,n^+,n^-}$ define linear operators 
\begin{equation*}
K_{g,n^+,n^-}: S_{n^-}(V)\longrightarrow S_{n^+}(V),
\end{equation*}
which are homogeneous of degree $2d_{g,n^+,n^-}$. 
\end{cor}
The operators $K_{g,n^+,n^-}$ naturally extend to $S(V)$ (by assuming that they act by zeros on $S_k(V)$ with $k\neq n^{-}$). Moreover, as we see in Section \ref{subsection_operator_boson}, an homogeneous operator always extends to the formal Fock space, then the $K_{g,n^+,n^-}$ induce endomorphisms
\begin{equation*}
K_{g,n^+,n^-}: \hat{S}(V)\longrightarrow \hat{S}(V).
\end{equation*}

\begin{rem}[Twists]
\label{remark_composition_1}
We change $V_{g,n^+,n^-}$ to $K_{g,n^+,n^-}$ for two reasons: first, the degree is twice the Euler characteristic in the second case, which is more natural and additive under composition. Moreover, the composition of two operators is given by the formula
\begin{equation}
K_{g_1,n^+,k}\circ K_{g_2,k,n^-}
=\prod_i L_i^+\int_{x\in |L^+|\cdot \Delta_{k}}\frac{ V_{g_1,n^+,k}(L^+|x) V_{g_2,k,n^-}(x|L^-)}{k!} \prod_j x_j d\sigma_{k}^{|L^+|}.
\end{equation}
Then we see that the \textit{twists} naturally appear in this formula; they correspond to the factor $\prod_j x_j$ in the integral \cite{mirzakhan2008}.
\end{rem}

\paragraph{Disconnected surfaces and disjoint union:}
It is natural to consider the volumes associated to disconnected surfaces. For $\Mol$ an oriented surface of type $(g,n^+,n^-)$ with labeled boundaries, we can use the notation $K_{\Mol}$ instead of $K_{g,n^+,n^-}$. If $\Mo=\sqcup_{c\in \pi_0(M)} \Mo(c)$ is a disconnected surface, the volume $K_{\Mol}$ satisfies the formula
\begin{equation*}
\label{formula_disconected_Kernel}
K_{\Mol}(L^+|L^-) = \prod_c K_{\Mo(c)}(L_{I^+(c)}^+|L_{I^-(c)}^-).
\end{equation*}
Where $I^\pm(c)$ are the two partitions of the positive and negative boundaries (see Remark \ref{remark_label_1}). If $\#\pi_0(\Mol)=k$, the function $K_{\Mol}$ is defined on the codimension $k$ affine subspace $\Lambda_{\Mol}=\prod \Lambda_{\Mo(c)}$ of $\R^{n^++n^-}$. To define the operator $K_{\Mol}$, we consider the set
\begin{equation*}
\Delta_{\Mol}(L^+)=\prod_c ~|L^+(c)|\cdot \Delta_{I^-(c)}=\{L^-\in \Rp^{n^-}~|~|L^+_{I^+(c)}|=|L^-_{I^-(c)}|,~~\forall~c~\in \pi_0(\Mo)\}
\end{equation*}
with the measure $d\sigma_{\Mol}^{L^+}=\prod_c d\sigma_{I^-(c)}^{|L^+_{I^+(c)}|}.$ And we define the operator $K_{\Mol}$ by
\begin{equation*}
K_{\Mol}\cdot P= \frac{1}{n^-!}\int_{L^-\in\Delta_{\Mol}(L^+)}K_{\Mol}(L^+|L^-) P(L^-)d\sigma_{\Mol}^{L^+}.
\end{equation*}
Using Fubini's Theorem, if we can factorize $P(L^-)=\prod_c P_c(L_{I^-(c)}^-)$, we can write
\begin{eqnarray*}
\int_{\Delta_{\Mol}(L^+)}K_{\Mol}(L^+|L^-)Pd\sigma_{\Mol}^{L^+}&=&\prod_c \int_{L^-\in |L^+(c)| \cdot \Delta_{n^-(c)}} K_{\Mol(c)}(L_{I^+(c)}^+|L^-)P_c(L^-)d\sigma_{n^-(c)}^{|L^-(c)|}\\
&=&\prod_c (K_{\Mol(c)}\cdot P_c)~(L_{I^+(c)}^+).
\end{eqnarray*}
By these formulas, $K_{\Mol}$ defines a linear operator
\begin{equation*}
K_{\Mol}~:~T(V)\longrightarrow T(V).
\end{equation*}
It is homogeneous of degree $2d(\Mo)$ and then extends to the completion. However, when $\Mol$ is not connected, $K_{\Mol}$ is not symmetric; we only have $K_{\Mol} \in \End(\hat{T}(V))$. By using linearity, we define an operator,
\begin{equation*}
\Kb : \M^{\s}\longrightarrow \End(\hat{T}(V)),
\end{equation*}
by $\Kb(e_{\Mol})=K_{\Mol}$. In Section \ref{subsection_alg_dir_surf}, we consider the action of the symmetric group by permuting the label and define the set of symmetric elements $S(\M^{\s})$. It admits a basis indexed by vectors $e_{\Mo}$, for $\Mo\in \bordos$ the set of unlabeled directed stable surfaces we set $K_{\Mo}=\Kb(e_{\Mo})$. It is, in some sense, the symmetrization of the operator $K_{\Mol}$ and corresponds to the sum of all the possible ways to label the boundaries of $\Mo$ up to homeomorphisms. This time $K_{\Mo}$ is a symmetric operator and extends to a linear operator on the formal Fock space,
\begin{equation*}
K_{\Mo}~:~\hat{S}(V)\longrightarrow \hat{S}(V).
\end{equation*}
\begin{rem}[Connected surfaces]
If $\Mo$ is connected of type $(g,n^+,n^-)$, there is only one possible way of labeling the boundaries, and then $K_{g,n^+,n^-}$ and $K_{\Mo}$ coincide in this case.
\end{rem}

In Section \ref{subsection_alg_dir_surf}, we define the disjoint union $\sqcup$ on $S(\M^{\s})$ (and $\hat{S}(\M^{\s})$), and we have the following natural proposition:
\begin{prop}
\label{prop_union_operator}
The operators satisfy the following rule:
\begin{equation*}
K_{\Mo_1\sqcup \Mo_2} = K_{\Mo_1}\sqcup K_{\Mo_2}
\end{equation*}
for each $\Mo_1,\Mo_2$. And then:
\begin{equation*}
\Kb~:~\hat{S}(\M^{\s})\longrightarrow \End(\hat{S}(V)),
\end{equation*}
defines a morphism of commutative algebras for the disjoint union.
\end{prop}
\begin{proof}
To prove Proposition \ref{prop_union_operator}, we can start with the following equality that follows from the definition:
\begin{equation*}
K_{\Mo_1\sqcup \Mo_2}(L^+|L^-) =\sum_{I^{\pm}_i} K_{\Mo_1}(L_{I_1^+}^+|L_{I_1^-}^-)K_{\Mo_2}(L_{I_2^+}^+|L_{I_2^-}^-).
\end{equation*}
where the sum is over all pairs of partitions $I_1^\pm\sqcup I_2^\pm=\{1,...,n^\pm_1+n^\pm_2\}$ in two sets of respective cardinals $n_1^\pm,n_2^{\pm}$. To conclude, we can integrate over a vector $e_\mu\in S(V)$ and use the Formula \ref{formule_union_op_1} for the union of two operators. As $K_{\Mo}$ is of degree $2d(\Mo)$ and  $S^d(\M^{\s})$ is finite dimensional; there are only a finite number of stable surfaces with a fixed Euler characteristic. Then, from the results of Section \ref{subsection_operator_boson} for $x=(x_d)\in \hat{S}(\M^{\s})$, the series $\sum_d \Kb(x_d)$ converges to an operator in $\End(\hat{S}(V))$.
\end{proof}

\paragraph{Unstable surfaces:}
The case of the cylinder (surface of type $(0,1,1)$) is quite fundamental. However, the volume $V_{0,1,1}$ of "quadrivalent" ribbon graphs on the cylinder is not defined, there is no ribbon graph. But we adopt the following definition:
\begin{equation*}
K_{0,1,1}(L^+|L^-)=\delta(L^+-L^-).
\end{equation*}
 The operator $K_{0,1,1}$ corresponds to the identity on $S_1(V)=V$ and then, to the projection:
\begin{equation*}
\pr_1~:~\hat{S}(V)\longrightarrow \hat{V}.
\end{equation*}
Using Lemma \ref{lemma_identity_pr}, we obtain:
\begin{lem}
\label{lem_exp_cyl_id}
The operator $\Kb$ is well defined on $\hat{S}(\M)$ and defines a morphism of commutative algebras
\begin{equation*}
\Kb~:~(\hat{S}(\M),\sqcup,+) \longrightarrow (\End(\hat{S}(V)),\sqcup,+).
\end{equation*}
Moreover, we have the following relation in $\End(\hat{S}(V))$
\begin{equation}
\label{formula_exp_K011}
\exp_\sqcup(K_{0,1,1})=\id.
\end{equation}
Finally, if $\Kb_{\Mo}=\Kb(\eb_{\Mo})$, then, it is a formal differential operators $\Kb_{\Mo}\in \Dh(V)$. And we have
\begin{equation*}
    \Kb_{\Mo_1\sqcup \Mo_2}=:\Kb_{\Mo_1} \Kb_{\Mo_2}:.
\end{equation*}
\end{lem}
\begin{proof}
The operator $K_{0,1,1}$ is of degree $0$ and then the spaces $S^d(\M)$ are infinite-dimensional; for $x\in \hat{S}(\M)$, the expression $\Kb(x)$ can contains an infinite number of terms of degree zeros. However, adding cylinders increases the number of boundary components, and this ensures convergence for the adic topology. Then the linear map $ \Kb$ is well defined on the completion $\hat{S}(\M)$. The second point is straightforward using the results of Lemma \ref{lemma_identity_pr}. As we prove that
\begin{equation*}
\exp_{\sqcup}(\pr_1)=\id.
\end{equation*}
We can write
\begin{equation*}
    \Kb_{\Mo}=\Kb(e_{\Mo}\sqcup \exp(e_{0,1,1}))=K_{\Mo}\sqcup \exp_{\sqcup}(K_{0,1,1})=K_{\Mo}\sqcup \id.
\end{equation*}
 Using Proposition \ref{prop_unstable_diff_operators}, we can conclude that $\Kb_{\Mo}$ is a formal differential operator.
\end{proof}

\paragraph{Ancestor potential:}
We can consider several operators; the first one, $K^{\s, \cf}$, is associated with connected and stable surfaces:
\begin{equation*}
K^{\s,\cf}={\sum}^{'}_{g,n^+,n^-}q^{2g-2+n^+ + n^-}K_{g,n^+,n^-}=\sum_{\Mo\in \bordosc} q^{d(\Mo)} K_{\Mo}.
\end{equation*}
The upper-script $'$ means that we impose the condition $2g-2+n^++n^->0$, $K^{\s, \cf}$ is well defined and is in $\End(\hat{S}(V))$. Similarly, we can also defines $K^{\cf},K^{\s}$ and $K$, where we sum respectively over the connected surfaces, the stable surfaces and all possible surfaces .

\begin{lem}
\label{lem_exp(Kc)=K}
The operators are related by the formula
\begin{equation*}
K^{\s}=\exp_{\sqcup}(K^{\s,\cf}),~~~\text{and}~~~K=\exp_{\sqcup}(K^{\cf}).
\end{equation*}
Moreover, we have 
\begin{equation*}
    K=K^{\s} \sqcup \id,
\end{equation*}
then $K$ is a formal differential operator on $\Q[[\tb]]$
\begin{equation}
K=\sum_{\mu^+,\mu^-} K^{\s}[\mu^+|\mu^-]\frac{t^{\mu^+}}{\mu^+!}\partial_{\mu^-}.
\end{equation}
Where $(K^{\s}[\mu^+|\mu^-])$ are the matrix coefficients of $K^{\s}$. Moreover, if $\Kb^{\cf}=K^{\cf}\sqcup \id$, then we have $K=:\exp(\Kb^{\cf}):$.
\end{lem}
\begin{proof}
By using Proposition \ref{prop_union_operator}, we have
\begin{equation*}
K_{\Mo_1 \sqcup \Mo_2}=K_{\Mo_1}\sqcup K_{\Mo_2}.
\end{equation*}
Moreover, by Formula \ref{formula_exp_surface_sqcup}, we can write
\begin{equation*}
\sum_{\Mo \in \bordos}e_{\Mo}=\exp_{\sqcup}( \sum_{\Mo \in \bordosc}e_{\Mo}).
\end{equation*}
Then we have
\begin{equation*}
K^{\s}=\Kb \left(\exp_{\sqcup}\left( \sum_{\Mo \in \bordosc}e_{\Mo}\right)\right)=\exp_{\sqcup}\left( \sum_{\Mo \in \bordosc}\Kb(e_{\Mo})\right)=\exp_{\sqcup}(K^{\s,\cf}).
\end{equation*}
As we have $K^{\cf}=K_{0,1,1}+K^{\cf,\s}$, then, by using Lemma \ref{lem_exp_cyl_id} we deduce 
\begin{equation*}
    K=\exp_{\sqcup}(K^{\cf,\s})\sqcup \exp_{\sqcup}(K_{0,1,1})=K^{\s}\sqcup id;
\end{equation*}
 Using Proposition \ref{prop_unstable_diff_operators}, we can conclude that $K$ is a formal differential operator.
\end{proof}

\subsection{Product of operators}
\paragraph{Operators associated with a directed acyclic stable graph:}
The composition of kernels corresponds to gluings of surfaces. But the map
\begin{equation*}
 \Kb: S(\M)\longrightarrow \End(S(V)),
\end{equation*}
is not a morphism for the composition. To obtain such properties, as we see in Paragraph \ref{paragraph_algebra_acyclic}, we need to introduce the operator $K_{\Gol}$ for $\Gol \in \acyclolp$, a labeled acyclic stable graph. To do that, we use the cone $\Lambda_{\Gol}$ of directed cycles in $\Gol$; see Paragraph \ref{paragraph_cone}. For a fixed $L=(L^+,L^-)\in \Lambda_{\partial \Gol}$, we consider the convex set $\Lambda_{\Gol}(L)$, with its "Lebesgue" measure $d\sigma_{\Gol}(L)$. Then, we can define the kernel
\begin{equation*}
 K_{\Gol}(L^+|L^-)=\int_{\Lambda_{\Gol}(L)}\prod_c K_{\Gol(c)}(L^+(c)|L^-(c))d\sigma_{\Gol}(L).
\end{equation*}
Similarly to the case of $K_{\Mol}$, we define the integral operator $K_{\Gol}$ with kernel given by $K_{\Go}(L^+|L^-)$. We also denote $K_{\Go}$ the symmetrization, for $\Go\in \acyclo$. As in the precedent case we have the following:
\begin{prop}
\label{prop_compo_Kb}
The operator $K_{\Go}$ preserves the space of symmetric polynomials and then induces a homogeneous linear operator of degree $2d(\Go)$. We can extend $\Kb$ to the free algebra $\hat{S}(\A)$ and define a linear map:
\begin{equation*}
\Kb~:~\hat{S}(\A)\longrightarrow \End(\hat{S}(V)).
\end{equation*}
The map $\Kb$ defines a representation of the algebra $\hat{S}(\A)$, and for each $\Go_1,\Go_2$ we have
\begin{equation*}
\Kb(e_{\Go_1}\cdot e_{\Go_2})=\Kb(e_{\Go_1})\circ \Kb(e_{\Go_2}) ,~~~\text{and}~~~\Kb(e_{\Go_1}\sqcup e_{\Go_2})=\Kb(e_{\Go_1})\sqcup \Kb(e_{\Go_2}).
\end{equation*}
\end{prop}

\section{Cut-and-Join equations}
\label{section_cut_and_join_operator}

In this part, we prove Theorem \ref{thm_cut_and_join} stated in the introduction.

\subsection{Recursion for the volumes of moduli space and acyclic decomposition}

We recall the recursion for the kernels $K_{g,n^+,n^-}(L^+|L^-)$ given in \cite{barazer2021cuttingorientableribbongraphs}. We use the notation $[x]_+=\max\{x,0\}$.
\begin{thm}
\label{thm_rec_quad_K}
Let $(g,n^+,n^-)$ such as $2g-2+n^++n^->1$, for all values of $(L^+,L^-)$, we have the following formula:
\begin{eqnarray*}
(2g-2+n^++n^-)K_{g,n^+,n^-}(L^+|L^-)&=& \frac{1}{2}\sum_{i\neq j} L_i^+L_j^+ ~K_{g,n^+-1,n^-}(L_i^+ +L_j^+,L^+_{\{i,j\}^c}|L^-)\\
&+&\sum_{i}\sum_{j} L_i^+ ~K_{g,n^+,n^--1}([L_i^+-L_j^-]_+,L^+_{\{i\}^c}|L^-_{\{j\}^c})\\
&+&\frac{1}{2}\sum_{i} L_i^+~\int_0^{L_i^+} K_{g-1,n^++1,n^-}(x,L_i^+-x,L^+_{\{i\}^c}|L^-) dx\\
&+&\frac{1}{2}\sum_{i}\sum_{\underset{I_1^{\pm} \sqcup I_2^{\pm}=I^{\pm}}{g_1+g_2=g}}  L_i^+~K_{g_1,n^+_1+1,n^-_1}([x_1]_+,L^+_{I_1^+}|L^-_{I_1^-})~ K_{g_2,n^+_2+1,n^-_2}([x_2]_+,L^+_{I_2^+}|L^-_{I_2^-}).
\end{eqnarray*}
Where the sum in the RHS contains only stable surfaces, and we use the notations:
\begin{equation*}
 x_l = [|L_{I_1^-}^-|-|L_{I_2^+}^+| ]_+,
\end{equation*}
and the sets $I^\pm$ correspond to $\{1,...,n^\pm\}$ minus the positive boundary $i$.
\end{thm}
The RHS of the formula in Theorem \ref{thm_rec_quad_K} corresponds to all the ways to glue a directed pair of pants of type $(0,1,2)$ or $(0,2,1)$.\\

A second result, which is equivalent to Theorem \ref{thm_rec_quad_K}, is the following:

\begin{thm}
\label{thm_acyclic_decomp}
    The functions $K_{g,n^=,n^-}$ are given by 
    \begin{equation*}
        K_{g,n^+,n^-}=\sum_{\Go\in \acyclogen_{g,n^+,n^-}}\frac{n_{\Go} K_{\Go}}{d_{g,n^+,n^-}!\#\Aut(\Go)},
    \end{equation*}
    where the sum run over all the acyclic pants decompositions and $n_{\Go}$ is defined in paragraph \ref{paragraph_exponential_structure}.
\end{thm}

\subsection{Cut-and-join equation for $K$}

We consider two pants gluing operators given by:
\begin{equation*}
P_{-}= K_{0,1,2} \sqcup id,~~~\text{and}~~~ P_{+}= K_{0,2,1}\sqcup id.
\end{equation*}
These operators correspond to adding a pair of pants with an arbitrary number of cylinders. We set:
\begin{equation*}
 W_1=P_++P_-.
\end{equation*}
We can easily compute the action of these operators:
\begin{lem}
\label{lem_action_P+_P-_1}
The action of the operators $P_{\pm}$ on $\hat{S}(V)$ is given by:
\begin{equation*}
P_{+}F(L)=\frac{1}{2}\sum_{i\neq j} L_iL_jF(L_i+L_j,L_{\{i,j\}^c}),~~~\text{and}~~~P_{-}F(L)= \frac{1}{2}\sum_i L_i \int_{0}^{L_i}F(x,L_i-x,L_{\{i\}^c})dx.
\end{equation*}
\end{lem}
\begin{proof}
 We have
\begin{equation*}
 K_{0,2,1}(L_1^+,L_2^+|L_1^-)=(L_1^+L_2^+)\delta(L_1^++L_2^+-L_1^-)~~~\text{and}~~~K_{0,1,2}(L_1^+|L_1^-,L_2^-)=L_1^+\delta(L_1^-+L_2^--L_1^+).
\end{equation*}
By using the formula for the union of two operators, we obtain the proposition.
\end{proof}
The recursion of Theorem \ref{thm_rec_quad_K} can be written in the following very simple way:

\begin{thm}[Cut-and-Join equation]
\label{thm_cut_and_join}
The operator $K(q)$ satisfies the following linear evolution called the Cut-and-Join equation
\begin{equation*}
\frac{d K}{dq}(q) = W_1 K(q).
\end{equation*}
With the initial condition $K(0)=id$, and then
\begin{equation*}
K(q)=\exp(qW_1).
\end{equation*}
\end{thm}

\begin{rem}
The variable $q$ is not essential here because it is half of the degree. Indeed, the last equation can be written as
\begin{equation*}
[D,K]=2W_1 K,~~~\text{or}~~~[D,K]=2P_+K+2P_-K.
\end{equation*}
\end{rem}

We first propose a straightforward proof.

\begin{proof}
There are several transformations to do in order to go from the formula in Theorem \ref{thm_rec_quad_K} to the Cut-and-Join equation. First of all, we put together terms of type $I,IV$ in the recursion; the kernel $K_{0,1,1}$ is constant equal to $1$, and then we can write
\begin{equation*}\sum_{i}\sum_{j} L_i^+ ~K_{g,n^+,n^--1}([L_i^+-L_j^-]_+,L^+_{\{i\}^c}|L^-_{\{j\}^c})=\sum_{i}\sum_{j} L_i^+ ~K_{g,n^+,n^--1}(x_1,L^+_{\{i\}^c}|L^-_{\{j\}^c})K_{0,1,1}(x_2|L^-_{j}).
\end{equation*}
 With the notations $x_1=[|L_{I_1^-}^-|-|L_{I_1^+}^+|]_+=[L_{i}^+-L_{j}^-]_+$ and $x_2=L_j^-=[|L_{I_2^-}^-|-|L_{I_2^+}^+|]_+$, the partitions are given by $I_1^+=\{i\}^c,I_1^-=\{j\}^c$ and $I_2^+=\emptyset,I_2^-=\{j\}$. Using these notations, the second and fourth lines fit together and lead to the new recursion:
 \begin{eqnarray*}
(2g-2+n^++n^-)K_{g,n^+,n^-}(L^+|L^-)&=& \frac{1}{2}\sum_{i\neq j} L_i^+L_j^+ ~K_{g,n^+-1,n^-}(L_i^+ +L_j^+,L^+_{\{i,j\}^c}|L^-)\\
&+&\frac{1}{2}\sum_{i} L_i^+~\int_0^{L_i^+} K_{g-1,n^++1,n^-}(x,L_i^+-x,L^+_{\{i\}^c}|L^-) dx\\
&+&\frac{1}{2}\sum_{i}\sum_{\underset{I_1^{\pm} \sqcup I_2^{\pm}=I^{\pm}}{g_1+g_2=g}} L_i^+~K_{g_1,n^+_1+1,n^-_1}(x_1,L^+_{I_1^+}|L^-_{I_1^-})~ K_{g_2,n^+_2+1,n^-_2}(x_2,L^+_{I_2^+}|L^-_{I_2^-}).
\end{eqnarray*}
This time the last term also includes unstable terms of type $(0,1,1)$. We can apply both sides of this equation to a vector $e_\mu\in S(V)$. The LHS is
\begin{equation*}
(2g-2+n^++n^-)K_{g,n^+,n^-}\cdot e_\mu.
\end{equation*}
Moreover, by using Lemma \ref{lem_action_P+_P-_1}, the first line becomes
\begin{eqnarray*}
&& \frac{1}{2~~n^-!}\sum_{i\neq j} (L_i^+L_j^+)~\int_{L^-\in |L^+|\Delta_{n^-}}K_{g,n^+-1,n^-}(L_i^+ +L_j^+,L^+_{\{i,j\}^c}|L^-)e_\mu(L^-)d\sigma_{n^-}^{|L^-|}\\
 &=& \frac{1}{2}\sum_{i\neq j} (L_i^+L_j^+)(K_{g,n^+-1,n^-}\cdot e_\mu) (L_i^+ +L_j^+,L^+_{\{i,j\}^c})\\
 &=&(P_+\circ K_{g,n^+-1,n^-})\cdot e_\mu.
\end{eqnarray*}

Similarly, we can treat the third line using the Fubini Theorem.
\begin{eqnarray*}
&&\frac{1}{2~~n^-!}\sum_{i} L_i^+~\int_{L^-\in |L^+|\cdot\Delta_{n^-}} \int_0^{L_i^+} K_{g-1,n^++1,n^-}(x,L_i^+-x,L^+_{\{i\}^c}|L^-)e_\mu(L^+) dx d\sigma_{n^-}^{|L^-|}\\
&=&\frac{1}{2}\sum_{i} L_i^+\int_0^{L_i^+} (K_{g-1,n^++1,n^-}\cdot e_\mu)(x,L^+_i-x,L^+_{\{i\}^c})dx\\
&=&(P_- \circ K_{g-1,n^++1,n^-})\cdot e_\mu.
\end{eqnarray*}
The last step is to recover the disconnected term. We first consider the contraction operator:
\begin{equation*}
C: \hat{S}(V)\otimes \hat{S}(V)\longrightarrow \hat{S}(V),
\end{equation*}
defined by
\begin{equation*}
C(f\otimes g)=\frac{1}{2}\sum_{i}\sum_{I_1^{\pm} \sqcup I_2^{\pm}=I^{\pm}} L_i^+~\int_0^{L_i^+}f(x,L^+_{I_1^+})~ g(L_i^+-x,L^+_{I_2^+})dx.
\end{equation*}
Secondly, for each partition $(I_1^-,I_2^-)$ we can write:
\begin{equation*}
e_\mu(L^-)=\sum_{\mu_1+\mu_2=\mu}e_{\mu_1}(L^-_{I_1^-})e_{\mu_2}(L^-_{I_2^-}).
\end{equation*}
Now, rewriting the last line of the recursion, we need to carefully compute the new domain of integration, and we obtain:
\begin{eqnarray*}
&&\frac{1}{2}\sum_{i}\sum_{\underset{I_1^{\pm} \sqcup I_2^{\pm}=I^{\pm}}{g_1+g_2=g}} \sum_{\mu_1+\mu_2=\mu} L_i^+~ \frac{n_1^-!n_2^-!}{n^-!}\int_{0}^{L_i^+}(K_{g_1,n^+_1+1,n^-_1}\cdot e_{\mu_1})(x,L^+_{I_1^+})~ (K_{g_2,n^+_2+1,n^-_2}\cdot e_{\mu_2})(x-L_i^+,L^+_{I_1^+})dx\\
&=& \frac{1}{2}\sum_{i}\sum_{\underset{I_1^{+} \sqcup I_2^{+}=\{i\}^c}{g_1+g_2=g}} \sum_{\mu_1+\mu_2=\mu} L_i^+~ \int_{0}^{L_i^+}(K_{g_1,n^+_1+1,n^-_1}\cdot e_{\mu_1})(x,L^+_{I_1^+})~ (K_{g_2,n^+_2+1,n^-_2}\cdot e_{\mu_2})(x-L_i^+,L^+_{I_1^+})dx\\
&=&\sum_{\mu_1+\mu_2=\mu} C((K_{g_1,n^+_1+1,n^-_1}\cdot e_{\mu_1}) \otimes (K_{g_2,n^+_2+1,n^-_2}\cdot e_{\mu_2})).
\end{eqnarray*}
The first equality is the result of integration; in the second, the inverse of the binomial coefficient absorbs the sum over $(I_1^-,I_2^-)$. To finish the proof, we use the following equality that shows how $P_-$ acts on disconnected objects:
\begin{lem}
\label{lem_action_P}
If $x^{\cf}=\sum_{g,n^+,n^-} x_{g,n^+,n^-}\in \hat{S}(V)$ and $x=\exp_{\sqcup}(x^{\cf}$), then we have
\begin{equation*}
P_-x=(\sum_{g,n^+,n^-} P_-x_{g,n^+,n^-} +\sum_{g_i,n^+_i,n^-_i} C (x_{g_1,n^+_1,n^-_1}\otimes x_{g_2,n^+_2,n^-_2}))\sqcup x.
 \end{equation*}
\end{lem}
The proof of Lemma \ref{lem_action_P} uses the fact that $K_{0,2,1}$ is a tensor of type $(2,1)$ and $K_{0,1,2}$ is of type $(1,2)$.
\end{proof}

We give a second proof that uses acyclic stable graphs,  we can rewrite Theorem \ref{thm_acyclic_decomp} in term of operators we still have 
\begin{equation*}
K_{g,n^+,n^-}=\underset{{\Go \in \acyclogen_{g,n^+,n^-}}}{\sum} \frac{n_{\Go} }{d_{g,n^+,n^-}! \#\Aut(\Go)}K_{\Go}.
\end{equation*}
Then using $\exp_\sqcup$ we have 
\begin{equation*}
\label{formula_K_exp}
K(q)=\sum_{\Go \in \acyclogen} q^{d(\Go)}\frac{n_{\Go}K_{\Go}}{d(\Go)!\#\Aut(\Go)}.
\end{equation*} 

\begin{proof}
In Proposition \ref{prop_exp_A} we state a formula that relate sums over acyclic stable graphs to exponential in the algebra $\hat{S}_{\infty}(\A)$.
\begin{equation*}
\exp(\eb_{0,2,1}+\eb_{0,1,2})=\sum_{\Go \in \acyclogen}\frac{n_{\Go}}{d(\Go)!\#\Aut(\Go)}\eb_{\Go}.
\end{equation*}

In Proposition \ref{prop_compo_Kb} we see that $\Kb$ is a morphism of algebra, we can apply it on both sides of this equation
\begin{eqnarray*}
\Kb(\eb_{0,2,1})=\Kb(e_{0,2,1})\sqcup \Kb(\exp_{\sqcup}(e_{0,1,1}))&=&K_{0,2,1}\sqcup id =P_+\\
\Kb(\eb_{0,1,2})=\Kb(e_{0,1,2})\sqcup \Kb(\exp_{\sqcup}(e_{0,1,1}))&=&K_{0,1,2}\sqcup id =P_- .
\end{eqnarray*}
And then
\begin{equation*}
\Kb(\exp(q(\eb_{0,2,1}+\eb_{0,1,2})))=\exp(\Kb(q(\eb_{0,2,1}+\eb_{0,1,2})))=\exp(qW_1).
\end{equation*}
Similarly, we also have
\begin{eqnarray*}
\Kb(\sum_{\Go \in \widetilde{\acyclogen}}\frac{n_{\Go}q^{d(\Go)}}{d(\Go)!\#\Aut(\Go)}e_{\Go})&=&\sum_{\Go \in \widetilde{\acyclogen}}\frac{n_{\Go}q^{d(\Go)}}{d(\Go)!\#\Aut(\Go)}\Kb(e_{\Go})\\ &=&\sum_{\Go \in \widetilde{\acyclogen}}\frac{n_{\Go}q^{d(\Go)}}{d(\Go)!\#\Aut(\Go)}K_{\Go}\\ &=& K(q).
\end{eqnarray*}
Where the last line is obtained with Formula \ref{formula_K_exp}. Finally, we obtain
\begin{equation*}
K=\exp(qW_1).
\end{equation*}
\end{proof}

\paragraph{Interpretation in terms of differential operators:}

In this part, we rewrite Theorem \ref{thm_cut_and_join} by using the isomorphism $\hat{S}(V)\simeq \Q[[\tb]]$, which was presented in Section \ref{subsection_boson}. Using the fact that
\begin{equation*}
K=K^{\s}\sqcup id,
\end{equation*}
and Proposition \ref{prop_unstable_diff_operators}, we already know the operator $K$ is expressed in terms of annihilation and creation operators. In this part, we give an explicit expression of the Cut-and-Join operators.

\begin{lem}
Operators $P_{-}$ and $P_{+}$ are given by the formula
\begin{equation*}
P_{-}=\frac{1}{2}\sum_{k,l}(k+l+2)t_{k+l+2}\partial_k\partial_l,~~~\text{and}~~~ P_{+}=\frac{1}{2}\sum_{k,l}(k+1)(l+1)t_{k+1}t_{l+1}\partial_{k+l}.
\end{equation*}
\end{lem}
\begin{proof}
By definition, $P_\pm$ can be written as $P_\pm^{\s}\sqcup id$ and then can be expressed in terms of creation and annihilation operators. Under the map to $\Q[[\tb]]$, they are formal differential operators. Moreover, using Proposition \ref{prop_unstable_diff_operators}, we know that the coefficient in front of $\tb^{\mu^+}\partial_{\mu^-}$ in $P_+$ ( resp $P_-$) is the matrix coefficient of $K_{0,2,1}$ (resp $K_{0,1,2}$). By using
\begin{equation*}
K_{0,2,1}\cdot f=L_1L_2f(L_1+L_2).
\end{equation*}
We obtain
\begin{equation*}
K_{0,2,1}\cdot e_k=\sum_{i+j=k}(i+1)(j+1)e_{i+1}\otimes e_{j+1}.
\end{equation*}
Which gives us the formula
\begin{equation*}
P_+=K_{0,2,1}\sqcup id=\frac{1}{2}\sum_{i+j=k}(i+1)(j+1)t_{i+1}t_{j+1}\partial_k.
\end{equation*}
Similarly, using
\begin{equation*}
K_{0,1,2}\cdot f=\frac{L_1}{2}\int_0^{L_1}f(x,L_1-x)dx,
\end{equation*}
we obtain
\begin{equation*}
K_{0,2,1}\cdot e_i\otimes e_j=\frac{1}{2}(i+j+2)e_{i+j+2}.
\end{equation*}
And then
\begin{equation*}
P_{-}=\frac{1}{2}\sum_{k,l}(i+j+2)t_{k+l+2}\partial_k\partial_l.
\end{equation*}
\end{proof}

\section{Partition function and combinatorial interpretations}
\label{section_partition_function}
\subsection{The partition function}
Let $(g,n^+,n^-)$, in this part, we consider the following integrals:
\begin{equation}
\label{formula_def_G_{g,n+,n-}}
Z_{g,n^+,n^-}(L)=\frac{1}{n^-!}\int_{L^-\in |L|\cdot \Delta_{n^-}} K_{g,n^+,n^-}(L|L^-)d\sigma_{n^-}^{|L|}.
\end{equation}
According to Proposition \ref{prop_transfert_lemma}, these functions are homogeneous symmetric polynomials of degree $4g-4+2n^++2n^-=2d_{g,n^+,n^-}$. In the case $(0,1,1)$ we have $ Z_{0,1,1}=1$, which corresponds to the amplitude of the cylinder. We can generalize this definition for disconnected surfaces in a straightforward way by using $K_{\Mo}$ and integrate over $\Delta_{\Mo}(L)$. In this case $Z_{\Mo}$ is also an homogeneous polynomial of degree $2d(\Mo)$. As in \ref{paragraph_operator_proj_lim} the vacuum is defined by
\begin{equation*}
 \ez=\exp_{\sqcup}(e_0).
\end{equation*}
It corresponds to the constant polynomial "in arbitrarily many variables". Then, we have
\begin{equation*}
Z_{\Mo}=K_{\Mo} \cdot \ez.
\end{equation*}
We can generalize this by linearity using the operator $\Kb$ defined in Section \ref{section_operator_volumes}. We can set
\begin{equation*}
\Zb(x)=\Kb(x)\cdot \ez,~~~\forall x\in \M.
\end{equation*}
As $\Zb$ is compatible with the graduation, then $\Zb$ extends to the formal completion $\hat{\M}$. Moreover, the image of a symmetric element $e_{\Mo}\in S(\M)$ is a symmetric "polynomial". Putting that together, we see that the operator $\Zb$ restricts to a linear map
\begin{eqnarray*}
\Zb: \hat{S}(\M)&\longrightarrow& \hat{S}(V) \\
x~~&\longrightarrow& \Kb(x)\cdot \ez.
\end{eqnarray*}
Moreover, using properties of $\ez$, we obtain the following proposition:
\begin{prop}
\label{prop_G_morphism}
The linear map $\Zb$ is a morphism of commutative algebras; we have $\Zb(x\sqcup y)=\Zb(x)\sqcup \Zb(y)$.
\end{prop}
\begin{proof}
Using Formula \ref{formule_union_op_1} in Section \ref{section_operator_volumes} , we have for each $e_{\Mo_1},e_{\Mo_2}\in S(\M)$
\begin{equation*}
Z_{\Mo_1\sqcup \Mo_2}= K_{\Mo_1\sqcup \Mo_1}\cdot \ez = (K_{\Mo_1}\sqcup K_{\Mo_1})\cdot \ez=(K_{\Mo_1}\cdot \ez) \sqcup (K_{\Mo_1}\cdot \ez)=Z_{\Mo_1}\sqcup Z_{\Mo_2}
\end{equation*}
\end{proof}
\paragraph{Partition function:}
The partition function is the image of the vacuum by $K$
\begin{equation*}
Z=K\cdot \ez== \sum_{\Mo} q^{d(\Mo)}Z_{\Mo}.
\end{equation*}
As before, we can also consider $Z^{\s},Z^{\cf}$ and $Z^{\cf,\s}$, which are associated with stable, connected surfaces and connected, stable surfaces. We derive the relation $K=\exp_{\sqcup}(K^{\cf})$ in Lemma \ref{lem_exp(Kc)=K} and using Proposition \ref{prop_G_morphism} it immediately implies the same relation for $Z$
\begin{equation*}
\label{formula_G_exp_Gc}
Z=\exp_{\sqcup}(Z^{\cf}).
\end{equation*}
Moreover, a stable operator is related to $K$ by the formula $K=K^{\s}\sqcup~id.$ And then we have the relation:
\begin{equation*}
Z=Z^{\s}\sqcup \ez.
\end{equation*}
Then $Z$ belongs to $\hat{S}_\infty(V)$ the space of symmetric polynomials in an infinite number of variables. In terms of formal series, by using the results of Section \ref{subsection_boson}, it is equivalent to
\begin{equation*}
Z(q,\tb)= \exp(t_0)~Z^{\s}(q,\tb^*).
\end{equation*}
Where $\tb^*=(t_1,t_2,...)$ is independent of $t_0$.

Using formula \ref{formula_G_exp_Gc}, the non-connected partition function is the exponential of the connected partition function,
\begin{equation*}
 Z(q,\tb)=\exp(Z^{\cf}(q,\tb)).
\end{equation*}
In this formula, the exponential is the one on the space of formal series. 
\begin{rem}
The variable $q$ is not essential; the operator $q^{\frac{D}{2}}$ corresponds to the change of variables $t_i \to q^{\frac{i}{2}}t_i$.
\end{rem}

\subsection{Combinatorial interpretations}
\paragraph{First combinatorial interpretation:}
Using Proposition \ref{prop_transfert_lemma}, we can find a combinatorial interpretation for the coefficients of the polynomial $Z_{g,n^+,n^-}$. For each $\alphab=(\alpha_1,...,\alpha_{n^+})$, let $\Ro_{g,n^+,n^-}(\alphab)$ be the number of directed ribbon graphs $\Ro$ counted with a weight equal to $\frac{1}{\#\Aut(\Ro)}$ and such that:
\begin{itemize}
\item The graph is quadrivalent.
\item There are $n^-$ negative unlabeled boundary components.
\item There are $n^+$ positive labeled boundary components, and the perimeter of the $i-th$ boundary is $\alpha_i$.
\end{itemize}
In the definition of $\Aut(\Ro)$, the automorphisms do not permute the positive boundaries.
\begin{prop}
\label{prop_combi_G}
The polynomial $Z_{g,n^+,n^-}$ is a generating function for the numbers $\Ro_{g,n^+,n^-}(\alphab)$, it satisfies
\begin{equation*}
 Z_{g,n^+,n^-}(L)=\sum_{\alphab} \Ro_{g,n^+,n^-}(\alphab) \prod_i \frac{L^{\alpha_i}_i}{(\alpha_i-1)!}.
\end{equation*}
\end{prop}
\begin{proof}
The polynomial $Z_{g,n^+,n^-}(L)$ is given by:
\begin{equation*}
\sum_{\Ro\in\rib_{g,n^+,n^-}}\sum_{\alphab^+} K_{g,n^+,n^-}[\alphab^+|0,...,0]\prod_{i=1}^{n^+}\frac{L_i^{\alpha_i^+}}{\alpha_i^+}.
\end{equation*}
Let $\Ro$ be a directed ribbon graph, in Proposition \ref{prop_transfert_lemma} we obtain the following expression:
\begin{equation*}
K_{\Ro}[\alphab^+|\alphab^-]=\prod_i \alpha_i^+ \frac{\overline{P}_{\Ro}(\alphab^+-\alphab^+(\Ro)|\alphab^-)}{n^-!~\#\Aut(\Ro)}.
\end{equation*}
We consider the case where the entries of $\alphab^-$ are zeros; moreover, using the biparticity, if $\overline{P}_{\Ro}(\alphab^+|\alphab^-)$ is non-zero, we must have $\alphab^+\ge \alphab^+(\Ro)$ and
\begin{equation*}
d(\alphab^+-\alphab^+(\Ro))=d(\alphab^-)=0,
\end{equation*}
 then we found $\alphab^+=\alphab^+(\Ro)$. The value of $\overline{P}_{\Ro}(0|0)$ is $1$ because there is only one integral point, we obtain
\begin{equation*}
K_{\Ro}[\alphab^+|0,...,0]= \prod_i \alpha_i^+ \left\{
\begin{array}{ll}
\frac{1}{\#\Aut(\Ro)} & \text{if} ~~\alphab^+=\alphab^+(\Ro)\\
~~~~~~0 & \text{else.}
\end{array}\right.
\end{equation*}
By summing over $\rib_{g,n^+,n^-}$, we can derive Proposition \ref{prop_combi_G}.
\end{proof}
\paragraph{Directed ribbon graphs and dessins d'enfants:}

There is numerous bijections between combinatorial structures on surfaces; the subject is rich, the literature is wide, and the knowledge of the author is quite poor. We present here some bijections that are meaningful in the context of directed ribbon graphs. The first one relates the directed ribbon graphs to the Grothendieck dessins d'enfants. We recall that Grothendieck dessins d'enfants are coverings over the sphere ramified over three points $(x_0,x_+,x_-)$. In particular, we can relate the count of directed quadrivalent ribbon graphs to Hurwitz numbers of dessins d'enfants. Let $h_{g,n^+,n^-}(\alphab)$ the Hurwitz number of dessins d'enfants with:
\begin{itemize}
\item Simple unlabeled ramifications above $x_0$.
\item $n^-$ unlabeled ramifications above $x_-$.
\item The ramifications above $x_+$ are labeled and specified by $\alphab$.
\end{itemize}
The bijection between directed ribbon graph and dessins d'enfants provides the following formula:
\begin{equation*}
\label{formula_R_h}
R^{\circ}_{g,n^+,n^-}(\alphab)=h_{g,n^+,n^-}(\alphab).
\end{equation*}

\paragraph{Dessin d'enfant to undirected ribbon graph:}
There is a second bijection that identifies dessins d'enfants and usual ribbon graphs. If a dessin d'enfant counted by $h_{g,n^+,n^-}(\alphab)$ is given, we can obtain a second graph by looking at preimages of the segment that connects $x_0$ and $x_+$. It is not necessarily a directed ribbon graph, and it satisfies the following properties:

\begin{itemize}
\item There is $n^+$ labeled boundary components of perimeters given by $2\alpha$.
\item The graph is bipartite; there is $n^-$ white vertices and $2g-2+n^++n^-$ bivalent black vertices.
\end{itemize}

Then we can forget black vertices because they are bivalent and the graph is bipartite; then, we obtain a ribbon graph with prescribed perimeters $\alphab$ for its boundary components and no restriction on the degree of its vertices (we allow univalent and bivalent vertices). Let $R_{g,n}(\alpha)$ be the number of such ribbon graphs weighted by $\frac{1}{\#\Aut(R)}$. We can forget the number of vertices and set 
\begin{equation*}
    R_{g}(\alphab)=\sum_n R_{g,n}(\alpha).
\end{equation*}
\begin{prop}
We have the relation
\begin{equation*}
\Ro_{g,n^+,n^-}(\alphab)=R_{g,n^-}(\alphab),
\end{equation*}
and $R_{g}(\alphab)=\sum_{n} \Ro_{g,n(\alphab),n}(\alphab)$.
\end{prop}
\begin{rem}[Description]
In this picture, the positive boundaries are the boundaries of the new graph; the quadrivalent vertices are mapped to bivalent vertices and then are deleted; and finally, the negative boundary components become the vertices of the new graph. This process is described in figure \ref{figure_bijection}.
\end{rem}
\begin{figure}
\centering
\includegraphics[height=3cm]{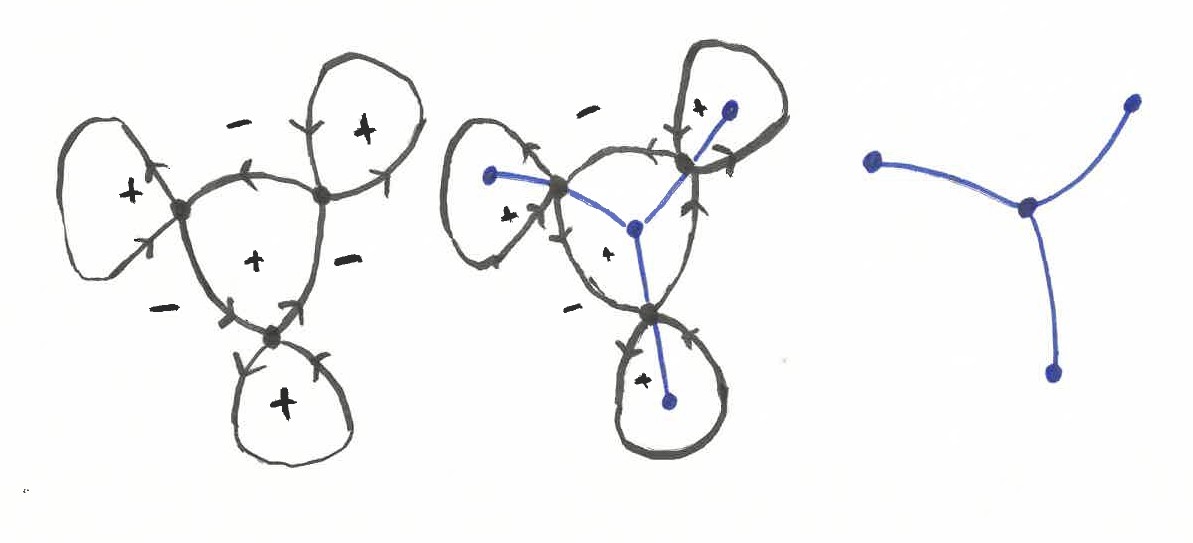}
\caption{Description of the bijection.}
\label{figure_bijection}
\end{figure}
\subsection{Cut-and-Join equation}
\paragraph{The generic case:}
In this part, we use Theorem \ref{thm_cut_and_join} to derive recursion for $Z$. Everything is straightforward.
\begin{prop}
As an element of $\Q[[\tb]]$, the partition function $Z$ satisfies the following linear equation:
\begin{eqnarray*}
\frac{\partial Z}{\partial q}(q) = \frac{1}{2}\sum_{k,l}(k+l+2)t_{k+l+2}\partial_i\partial_j Z(q) +\frac{1}{2}\sum_{k,l}(k+1)(l+1)t_{k+1}t_{l+1}\partial_{k+l} Z(q)
\end{eqnarray*}
\end{prop}
\begin{proof}
This is straightforward using the Cut-and-Join equation for $K$
\begin{equation*}
\frac{\partial K}{\partial q}= W_1K(q).
\end{equation*}
Applying this to the vacuum, we obtain: $\frac{\partial Z}{\partial q}= W_1Z.$
\end{proof}
\begin{rem}
As before, the variable $q$ is not necessary in this equation, as we have
\begin{equation*}
\frac{\partial Z}{\partial q}(1)=\sum_i it_i\partial_i Z(1).
\end{equation*}
\end{rem}
\begin{rem}
   In a similar way, the series $Z^{\cf}$ associated with connected surfaces also satisfies a recursion; the following fact is classical. The use of exponential makes this new equation non-linear:
 \begin{equation*}
\frac{\partial Z^{\cf}}{\partial q}= \frac{1}{2}\sum_{i,j}(i+1)(j+1)t_{i+1}t_{j+1}\partial_{i+j} Z^{\cf}+\frac{1}{2}\sum_{i,j}(i+j+2)t_{i+j+2}\partial_i\partial_jZ^{\cf} +\frac{1}{2}\sum_{i,j}(i+j+2)t_{i+j+2}\partial_i Z^{\cf}\partial_j Z^{\cf}.
\end{equation*}

\end{rem}

\paragraph{Surfaces with one negative boundary component:}
The kernels $K_{g,n^+,1}(L^+|L^-)$ are polynomials in $L^+$, and we can derive from Theorem \ref{thm_rec_quad_K} a recursion for them (see \cite{barazer2021cuttingorientableribbongraphs}). In this part, we recover these formulas by using the formalism of Fock space. We introduce a formal variable $t_-$, and the modified vacuum
\begin{equation*}
\ez(t_-)=\exp(t_-t_0)=t_-^N\cdot \ez.
\end{equation*}
Then we consider
\begin{equation*}
Z(q,t_-)=K(q)\cdot \ez(t_-).
\end{equation*}
And $Z(q, \tb,t_-)$ admits a development in the variable $t_-$,
\begin{equation*}
Z(q,\tb,t_-)=\sum_{n^-}Z^{n^-}(q,\tb)\frac{t_-^{n^-}}{n^-!}.
\end{equation*}
We can obtain the following result.
\begin{cor}
The generating series $Z^1$ also satisfies the Cut-and-Join equation
\begin{equation*}
\frac{\partial Z^1}{\partial q}=W_1Z^1.
\end{equation*}
But with the initial condition $ Z^1(0,t)=t_0.$
\end{cor}
The result is slightly different from the one given in \cite{barazer2021cuttingorientableribbongraphs} but are equivalent.
\begin{proof}
We can see that $Z(q,\tb,t_-)$ also satisfies the Cut-and-Join equation than $Z(q,\tb)$ because $Z(q,\tb,t_-)=K\cdot \ez(t_-)$, the initial condition is $\ez(t_-)$. We can collect the first order of the formula $\frac{\partial Z}{\partial q}=W_1Z$ we obtain $\frac{\partial Z^1}{\partial q}=W_1Z^1.$ We conclude by using $Z(0,\tb,t_-)=\exp(t_0t_-)$ and then $Z^1(0,\tb)=t_0$.
\end{proof}

\paragraph{Removing $t_0$ from the equation:}
Sometimes a Cut-and-Join equation can be presented without $t_0$. In our case we have 
\begin{equation*}
\frac{\partial Z}{\partial t_0}=Z~~~\text{and}~~~Z(\tb)=\exp(t_0)Z(0,\tb^*)=\exp(t_0)Z^*(t^*).
\end{equation*}
Then it is natural to consider the operators $P_0=\exp(-t_0)P\exp(t_0)$ acting on $\Q[[\tb^*]]$, we have the expression
\begin{equation*}
W_1'=\frac{1}{2}\sum_{i+j>0}(i+1)(j+1)t_{i+1}t_{j+1}\partial_{i+j}+\frac{1}{2}\sum_{i,j>0}(i+j+2)t_{i+j+2}\partial_{i}\partial_{j} + \sum_{i} (i+2)t_{i+2}\partial_{i}+\frac{t_1^2}{2}+t_2.
\end{equation*}
\begin{cor}
$Z^\s(q,\tb^*)$ is the solution of the linear homogeneous equation
\begin{equation*}
\frac{\partial Z^\s}{\partial q}=W_1'Z^\s ,
\end{equation*}
with the initial condition $Z^\s(0,\tb^*)=1$.
\end{cor}

\section{Laplace transforms, Tute equation  and topological recursion}
\label{section_TR_oriented}
\subsection{Loop equation}

The function $K_{g,n^+,n^-}$ is defined by the following formula:
\begin{equation*}
 K_{g,n^+,n^-}=\sum_{\Ro\in \rib_{g,n+,n^-}} \frac{K_{\Ro}}{\#\Aut(\Ro)}.
\end{equation*}
We have the following lemma.
\begin{lem}
\label{lem_refined_recursion}
The functions $K_{g,n^+,n^-}$ satisfy the relation:
\begin{eqnarray*}
\sum_{\Ro \in \rib_{g,n+,n^-}} \frac{\alpha_1^+(\Ro)K_{\Ro}}{\#\Aut(\Ro)}&=&\sum_{i \neq 1} L_1^+L_i^+K_{g,n^+-1,n^-}(L_1^++L_i^+,L^+_{\{1,i\}^c}|L^-)\\
&+&2\sum_{i} L_1^+ K_{g,n^+,n^--1}([L_1^+-L_i^-]_+,L^+_{\{1\}^c}|L^-_{\{i\}^c})\\
&+&L_1^+ \int_{0}^{L^+_1}K_{g-1,n^++1,n^-}(x,L^+_1-x,L^+_{\{1,i\}^c}|L^-)dx\\
&+&{\sum_{g_1,n^{\pm}_i,I_i^{\pm}}}^{'} L_1^+ K_{g_1,n^+_1+1,n^-_1}(x_1,L^+_{I_1^+}|L^-_{I_1^+})K_{g_2,n^+_2+1,n^-_2}(x_2,L^+_{I_2^+}|L^-_{I_2^+}).
\end{eqnarray*}
The notation $x_i$ was defined in Theorem \ref{thm_rec_quad_K}. And the upper script $'$ means that we exclude the unstable term $(g,n^+,n^-)=(0,1,1)$ in the sum.
\end{lem}

\begin{rem} We remark that this formula implies Theorem \ref{thm_rec_quad_K}. Using symmetry, Lemma \ref{lem_refined_recursion} is true for all boundary component $\beta_i^+$, then taking the sum over $i$ we can obtain the Theorem. We just need the formula:
\begin{equation*}
d(\alpha^+(\Ro))=\sum_i \alpha^+_i(\Ro)= 2d_{g,n^+,n^-}.
\end{equation*}
However, the last formula does not provides a recursion for piece-wise polynomials $K_{g,n^+,n^-}$. We don not know any simple way to express the LHS of the formula in terms of functions $K_{g,n^+,n^-}$. However, this formula allows us to compute recursively coefficients of $Z$.
\end{rem}
\begin{proof}
Lemma \ref{lem_refined_recursion} follows from results of \cite{barazer2021cuttingorientableribbongraphs}. If we fix $(g,n^+,n^-)$, it is possible to consider moduli space of generic directed metric ribbon graphs of type $(g,n^+,n^-)$ with a marked quadrivalent vertex on the first boundary $\beta_1^+$. If a vertex appears twice in the boundary $\beta_1^+$, we count it twice. This space admits a measure by taking the pullback of the Lebesgue measure by the covering map; moreover, the volume of this space is the LHS of the equation in Lemma \ref{lem_refined_recursion}. This space is mapped to the set of directed ribbon graphs with a marked vertex, but the map is neither surjective nor injective. Nevertheless, by applying the vertex surgery Theorem (see \cite{barazer2021cuttingorientableribbongraphs}), this vertex is contained in a unique admissible curve, which spares it from the rest of the graph, and $\beta_1^+$ is contained in the component that is removed. In the case where the gluings are of type $I$, there are only two possible vertices in the boundary of the corresponding pair of pants, and then the contribution is:
\begin{equation*}
2 \sum_jK_{g,n^+,n^--1}([L_1^+-L_i^-]_+,L^+_{\{1,i\}^c}|L^-).
\end{equation*}
In the case of type $II$, there are two vertices in the boundary, then the choice of the vertex removes the symmetry that exchanges the two other boundaries of the pants, and the contribution is:
\begin{equation*}
\int_{0}^{L^+_1}K_{g-1,n^++1,n^-}(x,L^+_1-x,L^+_{\{1,i\}^c}|L^-)dx.
\end{equation*}
We lose the factor $\frac{1}{2}$. In a similar way, we can obtain two other terms.
\end{proof}

\paragraph{The integral recursion for $Z$:}
We consider the formal series in $\hat{S}_{n}(V)$:
\begin{equation*}
Z_{g,n}(L)=\sum_{n^-}Z_{g,n,n-}(L).
\end{equation*}
Additionally, we also denote $Z^d_{n}$ the series associated with non-connected surfaces with an opposite Euler characteristic equal to $d$, and with $n$ positive boundary components. Using Lemma \ref{lem_refined_recursion}, we can derive the following proposition:
\begin{prop}
\label{prop_rec_G_3}
Functions $(Z_{n}^d)_{n,d}$ satisfy:
\begin{equation*}
\partial_1 Z_{n}^{d}(L)= \sum_{j\neq 1} L_jZ_{n-1}^{d-1}(L_1+L_j,L_{\{1,j\}^c})+ \int_0^{L_1}Z_{n+1}^{d-1}(x,L_1-x,L_{\{1\}^c})dx.
\end{equation*}
\end{prop}
\begin{rem}
As we have $Z_{n}^d(0,L_{\{1\}^c})=0$ if $d>0$ and $Z_{n}^0=1$, the last equation determines $Z_n^d$ with these two conditions.
\end{rem}

\paragraph{Tutte equations:}
From  Proposition \ref{prop_rec_G_3} we can obtain a Tutte recursion for the number of directed ribbon graphs \cite{tutte_1963, eynard2016counting}. We define $\Ro_{g,n}(\alpha)$ for all $(g,n)\neq (0,0)$ and all $\alpha\in \N^n$ by
\begin{equation*}
 \Ro_{g,n}(\alpha)= \sum_{n^-} \Ro_{g,n,n^-}(\alpha) = \Ro_{g,n^+,d(\alpha)-2d_{g,n}}(\alpha).
\end{equation*}
\begin{prop}
\label{prop_R(alpha)_tute}
Let $\Rot_{g,n}(\alphab)=\prod_i \alpha_i  \Ro_{g,n}(\alpha)$, we have the following relation:
\begin{eqnarray*}
\Rot_{g,n}(\alphab)&=& \sum_{i\neq 1}\alpha_i \Rot_{g,n-1}(\alpha_i+\alpha_1-2,\alpha_{\{1,i\}^c})\\
 &+& \sum_{k+l=\alpha_1-2}\Rot_{g-1,n+1}(k,l,\alphab_{\{1\}^c})+\sum_{I_1,I_2,g_1,g_2}\Rot_{g_1,n_1+1}(k,\alphab_{I_1})\Rot_{g_2,n_2+1}(l,\alphab_{I_2}).\\
\end{eqnarray*}
With the initialization $\Rot_{0,1}(0)=1$.
\end{prop}
\begin{proof}
By using Proposition \ref{prop_rec_G_3} that gives the recursion for $Z_{g,n}$ and also the formula:
\begin{eqnarray*}
Z_{g,n}(L)=\sum_{\alpha}\Rot_{g,n}(\alpha)\frac{L^\alpha}{\alpha!}.
\end{eqnarray*}
We can derive Proposition \ref{prop_R(alpha)_tute} by extraction of the coefficients.
\end{proof}
\begin{rem}[Combinatorial proof]
We note that Proposition \ref{prop_R(alpha)_tute} can be obtained in a combinatorial way by using bijection similarly to Tutte formulas and techniques described in \cite{eynard2016counting}.
\end{rem}

\paragraph{Laplace transform and Loop equation:}
We consider the Laplace transform $W_{g,n^+,n^-}$ of $Z_{g,n^+,n^-}$.
\begin{equation*}
W_{g,n^+,n^-}(x)=\int_{L}Z_{g,n^+,n^-}(L)e^{-\sum_i x_iL_i}dL=\sum_{\alpha} \Rot_{g,n^+,n^-}\prod_i\frac{1}{x_i^{\alpha_i+1}}.
\end{equation*}
We can also consider the formal Laplace transform of $Z_{g,n}$ denoted by $W_{g,n}$ and defined by
\begin{equation*}
W_{g,n}= \sum_{n^-}W_{g,n,n^-}\sum_\alpha \frac{\Rot_{g,n}(\alpha)}{x^{\alpha+1}}.
\end{equation*}

\begin{prop}
\label{prop_loop_eq}
The Laplace transform satisfies the recursion:
\begin{eqnarray*}
x_1 W_{g,n} &=& \sum_{i\neq 1} \frac{\partial}{\partial x_i}\left( \frac{W_{g,n-1}(x_1,x_{\{1,i\}^c})-W_{g,n-1}(x_i,x_{\{1,i\}^c})}{x_1-x_i}\right)\\
&+&W_{g-1,n+1}(x_1,x_1,x_{\{1\}^c})+ \sum_{I_1,I_2,g_1+g_2=g} W_{g_1-1,n_1+1}(x_1,x_{I_1})W_{g_2-1,n_2+1}(x_1,x_{I_2})+\delta_{g,0}\delta_{n,1}
\end{eqnarray*}
\end{prop}
The last equation is well know in the theory of matrix integrals and is called a Loop equation \cite{eynard2015random},\cite{eynard2016counting}.
\begin{proof}
To prove the result, it suffices to compute the Laplace transform in Proposition \ref{prop_rec_G_3}.
\end{proof}
\begin{rem}[Unstable term]
It is convenient to extract two terms of the sum over $(g,n)$, the ones that contain $W_{0,1}$. We can write:
\begin{eqnarray*}
(x_1-2W_{0,1}(x_1)) W_{g,n} &=& \sum_{i\neq 1} \frac{\partial}{\partial x_i}\left( \frac{W_{g,n-1}(x_1,x_{\{1,i\}^c})-W_{g,n-1}(x_i,x_{\{1,i\}^c})}{x_1-x_i}\right)\\
&+&W_{g-1,n+1}(x_1,x_1,x_{\{1\}^c})+ \underset{I_1,I_2,g_1+g_2=g}{{\sum}^{'}} W_{g_1-1,n_1+1}(x_1,x_{I_1})W_{g_2-1,n_2+1}(x_1,x_{I_2})+\delta_{g,0}\delta_{n,1}.
 \end{eqnarray*}
\end{rem}

\subsection{Topological recursion}
\paragraph{The amplitude $W_{0,1}$ of the disc:}
The series $W_{0,1}$ is associated to counting problems for directed ribbon graphs with only one positive boundary and drawn on the sphere. The only nontrivial term of the sequence of coefficients $(\Rot_{0,1,n}(\alpha))_{\alpha}$ is $\Rot_{0,1,n}(2n-2)$ and corresponds to graphs with $n$ labeled negative faces.\\

The recursion directly gives the following closed equation for $W_{0,1}$, which is classical in the theory of tree enumeration and in the theory of topological recursion for matrix integrals (see  \cite{eynard2016counting}):
\begin{equation*}
W_{0,1}^2-xW_{0,1}+1=0.
\end{equation*}

\begin{rem}[Trees]
This equation (and the series $W_{0,2}$) is very close to the one that characterizes the generating series of Catalan numbers and is related to counting trees. 
\end{rem}

To chose a solution of \ref{} we use the fact that $W_{0,1}$ is in $\C[[x^{-1}]]$ for large $x$ we obtain for $x>2$
\begin{equation*}
W_{0,1}=\frac{1}{2}\left( x-\sqrt{x^2-4}\right).
\end{equation*}
Expending $W_{0,1}$ in the variable $\frac{1}{x}$ leads to the formula
\begin{equation*}
 W_{0,1}=\sum_{n\ge 1} \frac{2^n(2n-3)!!}{2~n!}\frac{1}{x^{2n-1}}=\sum_{n\ge 1} \frac{1}{2(2n-1)}\binom{2n}{n}\frac{1}{x^{2n-1}},
\end{equation*}
which gives the formula for the coefficients
\begin{equation*}
    \Rot_{0,1,n}(2n-2)=\frac{1}{2~(2n-1)}\binom{2n}{n}.
\end{equation*}

\paragraph{Zukowsky variables and the universal expression for the cylinder amplitude:}
We give a classical derivation of the formula for $W_{0,2}$; this involves a change of variables by using the Zukowsky map; our main references are \cite{eynard2014short} or \cite{eynard2016counting}. These techniques have been developed by B.Eynard to solve the Loop/Tute equation.\\

The germ of $W_{0,1}$ is convergent in the neighborhood of $\infty$ in $\C P_1$, but due to the square root, it does not extend to a single-valued function on $\C P_1$ . Nevertheless, we can consider the smallest cover of $\C P_1$ on which the function is indeed single valued. In our case, there is two branch points at ${\pm 2}$, and then we might use a two-cover. The covering is explicitly given by the \textit{Zukowsky} map:
\begin{equation*}
x(z)=\left( z+\frac{1}{z}\right)~~~~z \in \C P^1.
\end{equation*}
The map $x$ is ramified at the two points $1,-1$, and the Galois involution is the inversion $\sigma(z)=\frac{1}{z}$. By abuse of notation, we still denote $W_{0,1}(z)$ the pullback of the germ $W_{0,1}(x)$ at $\infty$. From the formula for $W_{0,1}$, we derive
\begin{equation*}
W_{0,1}(z)=\frac{1}{z}~~~~\forall z \in \C P^1
\end{equation*}
at the neighborhood of $\infty$ and then on $\C P_1$ by analytic continuation. Then $W_{0,1}$ is meromophic on $\C P_1$. For $W_{0,2}$, we can obtain the following formula:
\begin{equation*}
W_{0,2}(z_1,z_2)x'(z_1)x'(z_2)=\frac{1}{(z_1z_2-1)^2}=\frac{1}{(z_1-z_2)^2}-\frac{x'(z_1)x'(z_2)}{(x(z_1)-x(z_2))^2}.
\end{equation*}
In a similar way, it is possible to define the series $W_{g,n}(z)=W_{g,n}(x(z))$, the composition by $x(z)$ makes sense in a neighborhood of $\infty$.\\

\paragraph{Differentials $\omega_{g,n}$}
Walking in the steps of B.Eynard and N.Orantin, we define differentials $\omega_{g,n}$ by
\begin{equation*}
\omega_{g,n}^0=W_{g,n}(z)\bigotimes_i x'(z_i) dz_i.
\end{equation*}
If $(g,n)\neq (0,2)$ and 
\begin{equation*}
\omega_{0,2}=\omega^0_{0,2}+\frac{x'(z_1)x'(z_2)dz_1\otimes dz_2}{(x(z_1)-x(z_2))^2}=\frac{dz_1\otimes dz_2}{(z_1-z_2)^2}.
\end{equation*}
is the \textit{Bergman Kernel}, or the fundamental second-kind differential of $\C P_1$. Each of these differentials defines a germ of multi-differentials near $\infty\in \C P_1$. If $y(z)=\frac{1}{z}$ we have $\omega_{0,1}=ydx$. And the equation:
\begin{equation*}
xy=y^2+1,
\end{equation*}
is called \textit{spectral curve}.

\paragraph{Topological recursion:}
We derive the topological recursion formula for $\omega_{g,n}$. We denote $\sigma_i^*\omega_{*,*}$ the pullback of $\omega_{*,*}$ by $\sigma_i$ acting in the $i-th$ variable and $d_i$ the differential of a function with respect to the $i-th$ variable.
\begin{thm}
\label{thm_tr_generic}
The $\omega_{g,n}$ are given by the following recursion:
\begin{eqnarray*}
\omega_{g,n}= \sum_{\epsilon=\pm 1}\oint_{z=\epsilon} \frac{dz_1}{\omega_{0,1}-\sigma^*\omega_{0,1}}\left (\frac{1}{z-z_1}-\frac{1}{z^{-1}-z_1}\right)&\cdot&(\sigma_2^* \omega_{g-1,n+1}(z,z,z_{\{1\}^c})\\
&+& \sum_{g_i,n_i,I_i}\omega_{g_1,n_1+1}(z,z_{I_1})\otimes \sigma_1^* \omega_{g_2,n_2+1}(z,z_{I_2})).
 \end{eqnarray*}
 And then, the differentials $\omega_{g,n}$ are solution of the E.O. topological recurssion with spectral curve 
 \begin{equation*}
     xy=y^2+1
 \end{equation*}
\end{thm}
\begin{rem}
\begin{itemize}
    \item It is not obvious that this formulation is indeed simpler. But there is an interesting phenomenon: the contribution of type $3$ gluing fits exactly with the one of type $2$ and disappears in the final formula. This is an explanation of why we change $\omega_{0,2}^0$ to $\omega_{0,2}$.
    \item There is a degree of freedom, the spectral curve is far from being unique. Indeed, the only term that contains $y$ is:
    \begin{equation*}
    \omega_{0,1}-\sigma^*\omega_{0,1}=(y-y\circ\sigma)dx=2y^- ,
    \end{equation*}
    $y^-$ is the anti-invariant part of $y$ under the involution. The invariant part of $y$ does not change the differentials $\omega_{g,n}$ for $(g,n)\neq (0,1)$. Changing $y$ by a term of the form $f(x)$ does not affect the recursion. There is a canonical choice given by taking an anti-invariant $y$, and it corresponds to $4y^2+4=x^2$.
\end{itemize}
\end{rem}

\begin{proof}
    We refer to \cite{eynard2016counting} for a proof.
\end{proof}

\subsection{Virasoro constraints}
We can derive from the Tutte equation a familly or relation called Virasoro constraint\footnote{ Indeed, these constraints are called Virasoro constraints, but we only consider the upper part of the Lie algebra, and then we shall call them \textit{Witt} constraints.}. These constraints are satisfied by the potential $Z$ and where also given in \cite{kazarian2015virasoro}. Let $L_i$ for $i\ge -1$ be the differential operators given by:
\begin{equation*}
L_i= - \partial_{i+2} + \sum_{j}(j+1) t_{j+1}\partial_{i+j+1}+\sum_{k+l=i} \partial_k\partial_l ,~~~~\forall i\ge -1.
\end{equation*}

\begin{prop}
  \label{prop_virasoro_alg}
  The generating series $Z$ satisfies the "Virasoro" constraints,
\begin{equation*}
L_i(Z)=0,~~~~~\forall~i\ge -1.
 \end{equation*}
The operators satisfy
\begin{equation*}
[L_i,L_j]=(i-j)L_{i+j}.
\end{equation*}  
\end{prop}
\begin{rem}
    \item The case $L_{-1}= -\partial_1 + \sum_{k}(k+1)t_{k+1}\partial_k$ is also called String equation.
    \item We can also add the constraint $C=-\partial_0+1$ and see that it commutes with all the $L_i$ and determine the series uniquely.
\end{rem}

\begin{proof}
    To prove Proposition \ref{prop_virasoro_alg}, we start with a reformulation of Proposition \ref{prop_R(alpha)_tute}. It admits a generalization for non-connected graphs, which is
\begin{eqnarray*}
{\Rt}_{n}^{\circ,d}(\alphab)&=& \sum_{i\neq 1}\alpha_i  {\Rt}_{n-1}^{\circ,d+1}(\alpha_i+\alpha_1-2,\alphab_{\{1,i\}^c})\\
&+& \sum_{k+l=\alpha_1-2} {\Rt}_{n+1}^{\circ,d-1}(k,l,\alphab_{\{1\}^c}).
\end{eqnarray*}
The non-connected coefficients $\Rt_{n}^{\circ,d}(\alpha)$ are symmetric and can be denoted $\Rt^{\circ}(\mu)$ for $\mu$ a partition (We can drop indices $d,n$ as we have $d=d(\mu)$ and $n=n(\mu)$). Let $\mu$ and $\alpha$ with $\mu(\alphab)=\mu$. For a fixed $i$, we can apply the last formula for all the boundary $j$ with $\alpha_j=i$. It leads to
    \begin{eqnarray}
\label{formula_R(mu)_tute}
i\mu(i)\frac{\Rt^{\circ}(\mu)}{\mu!} &=&\sum_j ij(\mu(i+j-2)+1)\frac{\Rt^{\circ}(\mu-\delta_i-\delta_j+\delta_{i+j-2})}{(\mu-\delta_i-\delta_j+\delta_{i+j-2})!}\\
&+&\sum_{k+l=i-2}i(\mu(k)+1)(\mu(l)+1)\frac{\Rt^{\circ}(\mu-\delta_i+\delta_k+\delta_l)}{(\mu-\delta_i+\delta_k+\delta_l)!}.
\end{eqnarray}
We have the expression:
\begin{equation*}
 Z =\sum_\mu \frac{{\Rt}_{n}^{\circ,d}(\mu)}{\mu!}\tb^\mu
 \end{equation*}
then we plug the recursion of Formula \ref{formula_R(mu)_tute} in this generating series, and this leads to the desired equation for $Z$:
 \begin{equation*}
 it_i\partial_i  Z =\sum_{j}ijt_j\partial_{i+j-2}Z+ \sum_{k+l=i-2}it_i\partial_{k}\partial_{l}Z.
\end{equation*}
We can obtain the commutation relation by a directed computation. 
\end{proof}

\section{Bivalent vertices}
\label{section_marked}
In this section, we discuss the case of ribbon graphs with bivalent vertices, which  is also of interest. Let $(g,n^+,n^-,m)$ satisfy $2g-2+n^+ +n^- +m>0$. We can define $\Mc_{g,n^+,n^-,m}$ as the combinatorial moduli space of directed metric ribbon graphs with $m$ marked vertices that are of degree $\ge 2$. Generically, the vertices at the marked points are of degree $2$. We define the volume $V_{g,n^+,n^-,m}$ using the natural Lebesgue measure. It is a continuous piecewise polynomial function of degree $4g-3+n^++n^-+m$. The statement of Proposition \ref{prop_transfert_lemma} remains valid in this context, leading to the following corollary:
\begin{cor}
    If $2g-2+n^+ +n^- +m>0$ the kernel $K_{g,n^+,n^-,m}$ defines a homogeneous operator on $\hat{S}(V)$ of degree $2d_{g,n^+,n^-}+m$.
\end{cor}

Analogously to Section \ref{section_operator_volumes}, we define operator
\begin{equation*}
    \Kbull(q_0,q_1)=\exp_{\sqcup}(\sum_{g,n^+,n^-,m} q_0^{m}q_1^{2g-2+n^++n^-}K_{g,n^+,n^+,m}))
\end{equation*}
As before, we write $\Kbull=K^{\bullet,\s}\sqcup id$, so that $\Kbull$ is a formal differential operator. We have the specialization:
\begin{equation*}
    \Kbull(0,q)=K(q).
\end{equation*}

\subsection{Recursion for $\Kbull$}

The results of \cite{barazer2021cuttingorientableribbongraphs} yield two different recursions for $V_{g,n^+,n^-,m}$; one involve removing a vertex of degree $4$ leading to the formula in Theorem \ref{thm_rec_quad_K}, and the other involves removing a vertex of degree $2$, which yields the following result:

\begin{prop}
    The piecewise polynomial kernel $K_{g,n^+,n^-,m}$ satisfies
    \begin{equation*}
        m K_{g,n^+,n^-,m}=EK_{g,n^+,n^-,m-1}.
    \end{equation*}
    Where $E$ is the function on $\Lambda_{n^+,n^-}$ defined by $E(L^+|L^-)=|L^+|=|L^-|$.
\end{prop}
From this, we immediately deduce:
\begin{equation*}
    K_{g,n^+,n^-,m}=\frac{E^m}{m!}K_{g,n^+,n^-}.
\end{equation*}
The multiplication by $E$ defines an operator on $\hat{S}(V)$, that will be denoted $W_0$. For $P\in S_n(V)$ we have
\begin{equation*}
    (W_0\cdot P)(L)=\sum_i L_i P(L).
\end{equation*}
We also observe that $W_0$ can be expressed as:
\begin{equation*}
    W_0=K_{0,1,1,1}\sqcup id,
\end{equation*}
which corresponds to gluing a directed ribbon graph on a cylinder with a single bivalent vertex. By Proposition \ref{prop_unstable_diff_operators}, this is a formal differential operator explicitly given by:
\begin{equation*}
    W_0=\sum_i (i+1)t_{i+1}\partial_i..
\end{equation*}
It is easy to verify that $E$ commutes with $P$ and also with all $K_{g,n^+,n^-,m}$ see (Appendix \ref{subsection_dual}). Thus, we conclude with the following:
\begin{prop}
    The series $\Kbull(q_0,q_1)$ satisfies the differential equations:
 \begin{equation*}
\frac{\partial \Kbull}{\partial q_0}=W_0\Kbull,~~~\text{and}~~~ \frac{\partial \Kbull}{\partial q_1}=W_1\Kbull.
\end{equation*}
    With the initial condition $\Kbull(0,0)=id$.
Moreover, the operators $P$ and $E$ commute, hence:
 \begin{equation*}
\Kbull(q_0,q_1)=\Kbull(q_0,0)\Kbull(0,q_1)=\exp(q_0W_0+q_1 W_1).
\end{equation*}
\end{prop}

\subsection{Partition function and combinatorial interpretation}

As in Section \ref{section_partition_function}, we can consider the polynomials $Z_{g,n^+,n^+,m}=  K_{g,n^+,n^-,m} \cdot \ez.$ These can be written as 
\begin{equation*}
    Z_{g,n^+,n^-,m}(L)=\sum_{\alphab} \Ro_{g,n^+,n^-,m}(\alphab) \prod \frac{L^{\alpha_i}_i}{(\alpha(i)-1)!}.
\end{equation*}
In this formula $\Ro_{g,n^+,n^-,m}(\alphab)$ is the number of directed ribbon graphs with $m$ bivalent vertices and all other vertices are of degree $4$. There is $n^+$ positive boundaries of perimeter $\alphab^+$, and $n^-$ unlabeled negative boundaries. We can form the partition function:
\begin{equation*}
    \Zbull(q_0,q_1)=\Kbull(q_0,q_1)\cdot \ez= \exp_{\sqcup}(\sum_{g,n^+,n^-,m} q_0^m q_1^{2g-2+n^++n^-} Z_{g,n^+,n^-,m}),
\end{equation*}
which satisfies: 
\begin{equation*}
    \frac{\partial \Zbull}{\partial q_0}=W_0\Zbull~~~\text{and}~~~ \frac{\partial \Zbull}{\partial q_1}=W_1\Zbull.
\end{equation*}
We can see that 
\begin{equation*}
    \exp(qW_0)\cdot \ez = \exp(\sum q^m t_m)
\end{equation*}
where the LHS is the exponential of a differential operator applied to $\ez=\exp(t_0)$ and the RHS is the exponential of a formal series. Hence, $K(q_0,q_1)$ solves the Cut-and-Join equation with the initial condition $\exp(\sum q_0^m t_m)$. 
    

\paragraph{Integral points:}

 We provide a final combinatorial bijection involving integral points in the moduli space $\M_{g,n^+,n^-}^{comb}$. Let $P_{g,n^+,n^-}^{\circ}(\alphab^+|\alphab^-)$ the number of integral points in the subspace $\M_{g,n^+,n^-}^{comb,*}(\alphab^+|\alphab^-)$, which is the number of integral points that are supported by directed ribbon graphs with vertices of degree $4$ vertices, and:
\begin{equation*}
P_{g,n^+,n^-}^{\circ}(\alphab^+)=\frac{1}{n^-!}\sum_{\alphab^-} P_{g,n^+,n^-}^{\circ}(\alphab^+|\alphab^-).
\end{equation*}
\begin{lem}
\label{lem_P_R}
We have:
\begin{equation*}
P_{g,n^+,n^-}^{\circ}(\alphab)=\Ro_{g,n^+,n^-,r}(\alphab)
\end{equation*}
with $r=d(\alphab)-4g+4-n^+-n^-$.
\end{lem}

\begin{proof}
This follows from the fact that bivalent vertices define a metric on the edges.
\end{proof}
This result allows us to relate generating series of integral points to $Z_{g,n^+,n^-,r}$. Let
\begin{equation*}
\mathbb{Z}_{g,n^+,n^-}(L)=\frac{1}{n^-!}\sum_{\alphab} P_{g,n^+,n^-}^{\circ}(\alphab) \frac{L^{\alpha}}{(\alphab-1) !}
\end{equation*}
Then the following formula holds:
\begin{prop}

The generating series $\mathbb{Z}_{g,n^+,n^-}$ and $Z_{g,n^+,n^-}$ are related by the following formula:
\begin{equation*}
\mathbb{Z}_{g,n^+,n^-}(L)=\exp(E(L))Z_{g,n^+,n^-}(L).
\end{equation*}
\end{prop}
\begin{proof}
Using Lemma \ref{lem_P_R}, we can obtain:
\begin{equation*}
\sum_{\alphab} P_{g,n^+,n^-}^{\circ}(\alphab) \frac{L^{\alphab}}{(\alphab-1) !}= \sum_{m} Z_{g,n^+,n^-,m}(L).
\end{equation*}
The polynomials $Z_{g,n^+,n^-,m}$ satifies the recurssion $mZ_{g,n^+,n^-,m}=E(L)Z_{g,n^+,n^-,m-1}$ and then we have 
\begin{equation*}
Z_{g,n^+,n^-,m}(L)=\frac{{E(L)}^m}{m!}Z_{g,n^+,n^-}(L).
\end{equation*}
Summing over $m$ gives the claimed identity.
\end{proof}
We can write the following relation
\begin{equation*}
\mathbb{Z}(q)=\Zbull(q,q).
\end{equation*}
which implies the following corollary:
\begin{cor}
The series $\mathbb{Z}$ satisfies the evolution equation
\begin{equation*}
\frac{\partial \mathbb{Z}}{\partial q}=(W_0+W_1) \mathbb{Z},
\end{equation*}
with initial condition $\mathbb{Z}(0)=\exp(t_0)$.
\end{cor}

\begin{rem}[Negative boundary components]
Unfortunately, we do not have control over the length of the negative boundary components. 
\end{rem}

\begin{appendices}
\section{Appendices}
\subsection{Duality and scalar product:}
\label{subsection_dual}

Functions $V_{g,n^+,n^-}$ have an additional symmetry that we can call times inversion. We have
\begin{equation*}
V_{g,n^+,n^-}(L^+|L^-)=V_{g,n^-,n^+}(L^-|L^+).
\end{equation*}
This is interesting for the following reason: we can consider, for instance, the scalar product on $S_n(V)$ given by
\begin{equation*}
(f,g)=\frac{1}{n!}\int_{\Rp^{n^+}} f(x)g(x)e^{-|x|}dx.
\end{equation*}
We then have the formula
\begin{equation*}
(f,V_{g,n^+,n^-}\cdot g) =\int_{\Lambda_{n^+,n^-}} f(L^+)V_{g,n^+,n^-}(L^+|L^-)g(L^-)e^{-|L^+|}\frac{d\sigma_{n^+,n^-}}{n^+!~n^-!}.
\end{equation*}
On $\Lambda_{n^+,n^-}$, we also have the relation $|L^+|=|L^-|$ by definition, and this implies that
\begin{eqnarray*}
\int_{\Lambda_{n^+,n^-}} f(L^+)V_{g,n^+,n^-}(L^+|L^-)g(L^-)e^{-|L^+|}\frac{d\sigma_{n^+,n^-}}{n^+!n^-!}
&=&\int_{\Lambda_{n^+,n^-}} f(L^+)V_{g,n^-,n^+}(L^-|L^+)g(L^-)e^{-|L^+|}\frac{d\sigma_{n^+,n^-}}{n^+!n^-!}\\
&=&\int_{\Lambda_{n^+,n^-}} g(L^-)V_{g,n^-,n^+}(L^-|L^+)f(L^+)e^{-|L^+|}\frac{d\sigma_{n^+,n^-}}{n^+!n^-!}
\end{eqnarray*}
And then we obtain the following nice formula:
\begin{equation*}
(f,V_{g,n^+,n^-}\cdot g)=(V_{g,n^-,n^+}\cdot f,g).
\end{equation*}
In particular, the operators $V_{g,n,n}$ are self-adjoint operators, and it might be possible to carry out their analysis. But we do not go deeper in this direction in this text. The time inversion then defines a structure of involutive algebra.

\subsection{Relation to Norbury polynomials and combinatorial interpretation of the Zukowsky map:}
\label{subsection_norbury}
The Norbury polynomials are Ehrhart quasi-polynomials that count the number of integral metric ribbon graphs. They have been introduced by P. Norbury in \cite{norbury2008counting} and studied also in \cite{norbury2013string}. Let $\Mc_{g,n,\Z}$ be the integral points in the combinatorial moduli spaces $\Mc_{g,n}$. For each $\alphab$, we denote $\Mc_{g,n,\Z}(\alphab)$ the subset of metric ribbon graphs such that $L_i(S)=\alpha_i$ for all $i\in \llbracket 1,n\rrbracket $. Then the Norbury polynomial is given by
\begin{equation*}
N_{g,n}(\alphab)=\sum_{S\in \M^{comb}_{g,n,\Z}(\alphab)} \frac{1}{\#\Aut(S)}.
\end{equation*}
The generating function for these polynomials is the Ehrhart series; it's given by
\begin{equation*}
 F^{comb}_{g,n}(y)=\sum_{g,n} N_{g,n}(\alphab)y^{\alphab}
\end{equation*}
and defined in \cite{norbury2013string}. The aim of this paragraph is to give a proof of a surprising and beautiful formula that relies on the series $F^{comb}_{g,n}$ and $ W^*_{g,n}$. As in \cite{norbury2013string}, we also denote 
\begin{equation*}
    \Omega_{g,n}=d_1 ...d_n N_{g,n}.
\end{equation*}
Where $d_i$ is the differential with respect to the $i-th$ variable.
\begin{thm}
\label{thm_W_norbury}
The generating series $F^{comb}_{g,n}$ and $W^*$ are related by the following formula:
\begin{equation*}
   F^{comb}_{g,n}(u(x))=W^*_{g,n}(x).
\end{equation*}
Where  $u$ is given by 
\begin{equation*}
    u(x)=\frac{x-\sqrt{x^2-4}}{2}.
\end{equation*}
Consecutively, we have the relation
\begin{equation}
\label{formula_norbury_1}
 W_{g,n}\otimes_i dx_i= u^* \Omega_{g,n}.
\end{equation}
\end{thm}

\begin{proof}
Let $\rib_{g,n,*}$ be the set of ribbon graphs with $n$ faces and possibly univalent vertices; let also $\rib_{g,n}'$ be the set of ribbon graphs with no univalent vertices. There is a  a contraction map
\begin{equation*}
  \text{Trunc} :  \rib_{g,n,*}\longrightarrow \rib_{g,n},
\end{equation*}
which is, for instance, described in \cite{okounkov2000random}, or in figure \ref{figure_trunc}. This map is a fibration, fiber consists of trees planted around vertices of the graphs. We can choose to conserve the bivalent vertices that are in the trunc of the graph. These bivalent vertices define a metric on the ribbon graph, and then this induces a surjection
\begin{equation*}
    \rib_{g,n,*}\longrightarrow \M^{comb}_{g,n,\Z}.
\end{equation*}
Where $\M^{comb}_{g,n,\Z}$ is the set of all  integral metric ribbon graphs and the set of integral points in the combinatorial moduli space. This defines a new map 
\begin{equation*}
     \text{Trunc}' : \rib_{g,n}\longrightarrow \M^{comb}_{g,n,\Z}.
\end{equation*}

For each $S$ in the image, we can consider the formal series
\begin{equation*}
  T_S=  \sum_{R,\text{Trunc}'(R)=S}\frac{1}{\#\Aut(R)} \prod_i \frac{1}{x_i^{\alpha_i(R)}},
\end{equation*}
and then we have the push-forward formula:
\begin{equation}
\label{formula_trunc_1}
\sum_{R,Trunc'(R)=S}\frac{1}{\#\Aut(R)} \prod_i \frac{1}{x_i^{\alpha_i(R)}} = \sum_S \frac{T_S}{\#\Aut(S)}.
\end{equation}
Now we prove the following proposition:
\begin{lem}
Near $\infty$, we have the formula 
\begin{equation*}
    T_S=\prod_i u(x_i)^{L_i(S)},
\end{equation*}
where $L_i$ is the combinatorial length of the $i-th$ boundary.
\end{lem}

\begin{proof}
Using figure \ref{figure_trunc}, we see that the fiber corresponds to a planar tree planted at each "corner" of the boundary. In the generating series, we add to remember the perimeter of the boundary, and then each tree $T$ glued to the boundary $i$ is weighted by $\frac{1}{x_i^{2\#X_1T+1}}$ where $X_1T$ is the set of edges in the boundary. Because each edge in the tree contributes twice to the boundary length, we also include the edge on the "left" of the tree. Then, to compute the preimage, we start by computing the generating series associated with a single corner in a given boundary $i$. Let $u_0$ be the series
\begin{equation*}
    u_0(x)=\sum_T \frac{1}{x^{2\#X_1T}}.
\end{equation*}
The sum is over all the planar tree with a marked vertex. This generating series is easily computable; indeed, by removing the roots, we can derive the following recurrence relation:
\begin{equation*}
    u_0(x)=1+\frac{u_0(x)}{x^2}+\frac{u_0(x)^2}{x^4}+...=\frac{x^2}{x^2-u_0(x)}.
\end{equation*}
And then 
\begin{equation*}
    u_0(x)^2-x^2u_0(x)+x^2=0.
\end{equation*}
Solving this equation leads to the formula 
\begin{equation*}
    u_0(x)=\frac{x^2-\sqrt{x^4-4x^2}}{2},
\end{equation*}
where we use $u_0(x)=1+o(1)$. Then we obtain 
\begin{equation*}
    u(x)=\frac{u_0(x)}{x}=\frac{x-\sqrt{x^2-4}}{2}.
\end{equation*}
Then, to prove the lemma, the contribution of each corner is given by $u(x)$, and the contribution of the boundary $i$ in $S$ is 
\begin{equation*}
     u(x)^{L_i(S)}.
\end{equation*}
Because there are exactly $L_i(S)$ corners in the boundary $i$. As the contribution of the boundaries is independent, we obtain
\begin{equation*}
    T_S(x)=\prod_i u(x)^{L_i(S)}.
\end{equation*}
\end{proof}

Using this lemma, we finish the proof of \ref{thm_W_norbury}. It suffices to substitute the expression of $T_S(x)$ in the formula \ref{formula_trunc_1} to find $F_{g,n}^{comb}(u(x))$. We have in one hand
\begin{equation*}
    W_{g,n^+,n^-}^*(x)=\int_{\Lambda_{n^+,n^-}}V_{g,n^+,n^-},(L^+|L^-)\exp(\sum_i x_i L^+_i)d\sigma_{n^+,n^-}
\end{equation*}
and, on the other hand 
\begin{equation*}
   K_{g,n^+,n^-}(L^+|L^-)=\prod_i L_i^+ V_{g,n^+,n^-}(L^+|L^-).
\end{equation*}
Which leads to
\begin{equation*}
  W_{g,n^+,n^-}(x)=\partial_1...\partial_{n^+}  W_{g,n^+,n^-}^*(x).
\end{equation*}
Summing over $n^-$ gives the second claim of the theorem \ref{thm_W_norbury}.
\end{proof}

\begin{figure}[h!]
\centering
    \includegraphics[width=6cm]{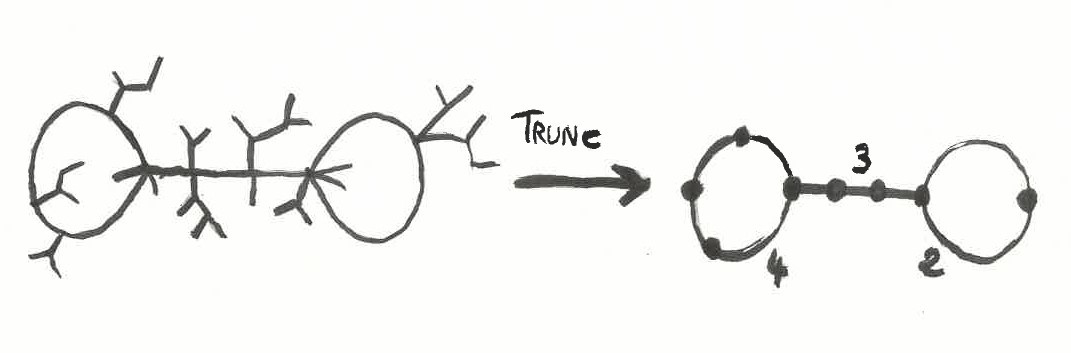}
    \caption{Effect of the map, Trunc.}
    \label{figure_trunc}
\end{figure}

\end{appendices}
\bibliography{biblio_tot}
\bibliographystyle{alpha}

\end{document}